\documentclass{aic}

\aicAUTHORdetails{%
  title = {Induced subgraphs and tree decompositions\\
  XIII. Basic obstructions in $\mathcal{H}$-free graphs for finite $\mathcal{H}$}, 
  author = {Bogdan Alecu, Maria Chudnovsky, Sepehr Hajebi, and Sophie Spirkl},
  plaintextauthor = {Bogdan Alecu, Maria Chudnovsky, Sepehr Hajebi, Sophie Spirkl},
  plaintexttitle = {Induced subgraphs and tree decompositions XIII. Basic obstructions in H-free graphs for finite H}, 
  runningtitle = {Induced subgraphs and tree decompositions XIII.}, 
  copyrightauthor = {B. Alecu, M. Chudnovsky, S. Hajebi and S. Spirkl},
  keywords = {Tree decompositions, Treewidth, Induced subgraphs, graph minors.},
}   

\aicEDITORdetails{%
   year={2024},
   number={6},
   received={26 November 2023},   
   published={29 November 2024},  
   doi={10.19086/aic.2024.6},      
}   

\usepackage[shortlabels]{enumitem}
\usepackage{mathdots}
\usepackage{dsfont}

\newtheorem{theorem}{Theorem}[section]
\newtheorem{lemma}[theorem]{Lemma}

\newtheorem{corollary}[theorem]{Corollary}
\newtheorem{question}[theorem]{Question}
\newtheorem{observation}[theorem]{Observation}

\DeclareMathOperator{\tw}{tw}

\def\dd{\hbox{-}}   

\newcommand{\mf}{\mathfrak}
\newcommand{\mca}{\mathcal}
\newcommand{\poi}{\mathbb{N}}

\newcounter{tbox}
\newcommand{\sta}[1]{\medskip\medskip\refstepcounter{tbox}\noindent{\parbox{\textwidth}{(\thetbox) \emph{#1}}}\vspace*{0.3cm}}
\newcommand{\mylongtitle}[1]{%
  \ifodd\value{page}%
    \protect\parbox{0.97\linewidth}{#1}\hfill%
  \else%
    \hfill\protect\parbox{0.97\linewidth}{#1}%
  \fi%
}

\begin{document}

\begin{frontmatter}[classification=text]

\title{Induced subgraphs and tree decompositions\\
XIII. Basic obstructions in $\mathcal{H}$-free graphs for finite $\mathcal{H}$} 

\author[alecu]{Bogdan Alecu\thanks{Supported by DMS-EPSRC Grant EP/V002813/1.}}
\author[chudnov]{Maria Chudnovsky\thanks{Supported by NSF-EPSRC Grant DMS-2120644 and by AFOSR grant FA9550-22-1-0083.}}
\author[hajebi]{Sepehr Hajebi}
\author[spirkl]{Sophie Spirkl\thanks{We acknowledge the support of the Natural Sciences and Engineering Research Council of Canada (NSERC), [funding reference number RGPIN-2020-03912].
Cette recherche a \'et\'e financ\'ee par le Conseil de recherches en sciences naturelles et en g\'enie du Canada (CRSNG), [num\'ero de r\'ef\'erence RGPIN-2020-03912]. This project was funded in part by the Government of Ontario. This research was conducted while Spirkl was an Alfred P. Sloan Fellow.}}

\begin{abstract}
Unlike minors, the induced subgraph obstructions to bounded treewidth come in a large variety, including, for every $t\in \poi$, \textit{the $t$-basic obstructions}: the  graphs
$K_{t+1}$ and $K_{t,t}$, along with the subdivisions of the $t$-by-$t$ wall and their line graphs. But this list is far from complete. The simplest example of a ``non-basic'' obstruction is due to Pohoata and Davies (independently). For every $n \in \poi$, they construct certain graphs of treewidth $n$ and with no $3$-basic obstruction as an induced subgraph, which we call \textit{$n$-arrays}.

Let us say a graph class $\mathcal{G}$ is \textit{clean} if the only obstructions to bounded treewidth in $\mathcal{G}$ are in fact the basic ones. It follows that a full description of the induced subgraph obstructions to bounded treewidth is equivalent to a characterization of all families $\mathcal{H}$ of graphs for which the class of all $\mathcal{H}$-free graphs is clean (a graph $G$ is \textit{$\mathcal{H}$-free} if no induced subgraph of $G$ is isomorphic to any graph in $\mathcal{H}$). 

This remains elusive, but there is an immediate necessary condition: if $\mathcal{H}$-free graphs are clean, then there are only finitely many $n\in \poi$ such that there is an $n$-array which is $\mathcal{H}$-free. The above necessary condition is not sufficient in general. However, the situation turns out to be different if $\mathcal{H}$ is \textit{finite}: we prove that for every finite set $\mathcal{H}$ of graphs, the class of all $\mathcal{H}$-free graphs is clean if and only if there is no $\mathcal{H}$-free $n$-array except possibly for finitely many values of $n$.
\end{abstract}
\end{frontmatter}


\section{Introduction}\label{sec:intro}

The set of all positive integers is denoted by $\poi$, and for every integer $n$, we write $\poi_n$ for the set of all positive integers less than or equal to $n$ (so $\poi_n=\emptyset$ if $n\leq 0$). Graphs in this paper have finite vertex sets, no loops and no parallel edges. Let $G=(V(G),E(G))$ be a graph.  An \textit{induced subgraph} of $G$ is the graph $G\setminus X$ for some $X\subseteq V(G)$, that is, the graph obtained from $G$ by removing the vertices in $X$. For $X\subseteq V(G)$, we use both $X$ and $G[X]$ to denote the subgraph of $G$ induced on $X$, which is the same as $G\setminus (V(G)\setminus X)$. We say $G$ \textit{contains} a graph $H$ if $H$ is isomorphic to an induced subgraph of $G$; otherwise, we say $G$ is \textit{$H$-free}. We also say $G$ is \textit{$\mathcal{H}$-free} for a family $\mathcal{H}$ of graphs if $G$ is $H$-free for all $H\in \mathcal{H}$.

The \textit{treewidth} of a graph $G$ (denoted by $\tw(G)$) is the smallest $w\in \poi$ for which $G$ can be represented as (a subgraph of) an intersection graph of subtrees of a tree, such that each vertex of the underlying tree appears in at most $w+1$ subtrees. For instance, the ``Helly property of subtrees'' \cite{helly} implies that for every $t\in \poi$, the complete graph $K_{t+1}$ and the complete bipartite graph $K_{t,t}$ both have treewidth $t$.

There are also sparse graphs with arbitrarily large treewidth, the most well-known example of which is the \textit{hexagonal grid} (see Figure~\ref{fig:hexwall}). For every $t\in \poi$, the $t$-by-$t$ hexagonal grid, also called the \textit{$t$-by-$t$ wall}, denoted $W_{t\times t}$, has treewidth $t$ \cite{GMV}, and the same remains true for all subdivisions of $W_{t\times t}$, as well. Indeed,  Robertson and Seymour proved in 1986 \cite{GMV} that containing a hexagonal grid as a minor (or a subdivided one as a subgraph) is qualitatively the only reason why a graph may have large treewidth:

\begin{figure}[t!]
    \centering
    \includegraphics[scale=0.6]{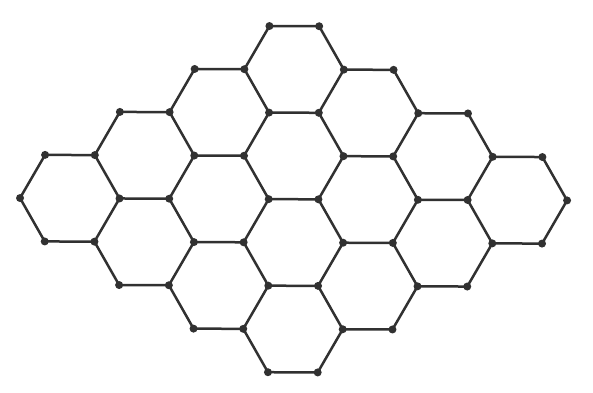}
    \caption{The $5$-by-$5$ hexagonal grid $W_{5\times 5}$.}
    \label{fig:hexwall}
\end{figure}

\begin{theorem}[Robertson and Seymour \cite{GMV}]\label{wallminor}
For every $t\in \poi$,
every graph of sufficiently large treewidth contains
a subdivision of $W_{t\times t}$ as a subgraph.
\end{theorem}

For induced subgraphs, however, excluding just the walls is not enough to guarantee bounded treewidth: note that complete graphs, complete bipartite graphs and subdivided walls are three pairwise ``independent'' types of graphs with arbitrarily large treewidth, in the sense that no graph from one type contains an induced subgraph of another type with large treewidth. There is also a fourth example, namely the line-graphs of subdivided walls, where the {\em line-graph} $L(F)$ of a graph $F$ is the graph with vertex set $E(F)$, such that two vertices of $L(F)$ are adjacent if the corresponding edges of $F$ share an end. 

It is useful to group all these graphs together: given $t\in \poi$, we say a graph $H$ is a \textit{$t$-basic obstruction} if $G$ is isomorphic to one the following: the complete graph $K_{t+1}$, the complete bipartite graph $K_{t,t}$, a subdivision of $W_{t\times t}$, or the line-graph of a subdivision of $W_{t\times t}$  (see Figure~\ref{fig:3basic}). For every $t\in \poi$, every $t$-basic obstruction has treewidth $t$, and each induced subgraph of large treewidth in a basic obstruction of a given type contains a basic obstruction of the same type and of (relatively) large treewidth. Let us say a graph $G$ is \textit{$t$-clean} if $G$ contains no $t$-basic obstruction. It follows that every graph of treewidth less than $t$ is $t$-clean.

\begin{figure}[t!]
\centering
\includegraphics[scale=0.6]{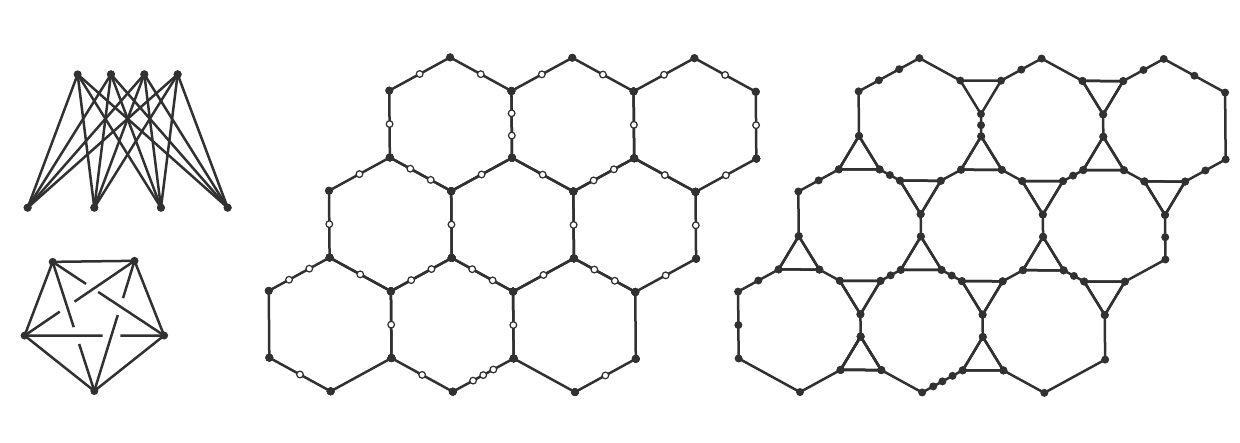}
\caption{The $4$-basic obstructions, including a subdivision of $W_{4\times 4}$ (middle) and its line-graph (right).}
\label{fig:3basic}
\end{figure}

One may hope for the basic obstructions to be the only induced subgraph obstructions to bounded treewidth. In other words, the neatest possible analog of Theorem~\ref{wallminor} for induced subgraphs would be the following: \textit{for every $t\in \poi$,  there is a constant $n=n(t)\in \poi$ such that every $t$-clean graph has treewidth less than $n$.} However, there are now several counterexamples to this statement \cite{deathstar, Davies2, Pohoata,layered-wheels}. The simplest construction is due to Pohoata and Davies (independently) \cite{ Davies2, Pohoata}, consisting of $3$-clean graphs of arbitrarily large treewidth. We call these graphs \textit{arrays}. 

Specifically, for $n\in \poi$, an $n$-array is a graph consisting of $n$ pairwise disjoint paths $L_1,\ldots, L_n$ with no edges between them as well as $n$ pairwise non-adjacent vertices $x_1,\ldots, x_n$, such that for every $i\in \poi_n$:
\begin{itemize}
    \item each of $x_1,\ldots, x_n$ has \textbf{at least} one neighbor in $L_i$; and
    \item for every $j\in \poi_{n-1}$, all neighbors of $x_j$ in $L_i$ appear before all neighbors of $x_{j+1}$ in $L_i$ (in particular, each vertex of $L_i$ is adjacent to at most one of $x_1, \dots, x_n$).
\end{itemize}
\begin{figure}[t!]
\centering
\includegraphics[scale=0.7]{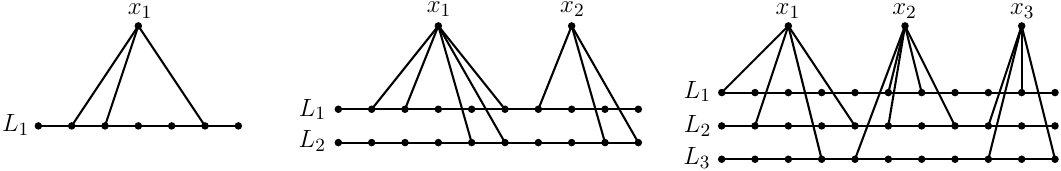}
\caption{Left to right: Examples of $n$-arrays for $n\in \{1,2,3\}$.}
\label{fig:davies}
\end{figure}
See Figure~\ref{fig:davies}. As mentioned earlier, arrays are $3$-clean graphs with unbounded treewidth:

\begin{theorem}[Pohoata \cite{Pohoata}, Davies \cite{Davies2}; see also Theorem 3.1 in \cite{twxii}]\label{thm:arrayproperties}
    For every $n\in \poi$, every $n$-array is a $3$-clean graph of treewidth at least $n$.
\end{theorem}

This motivates the following definition: A graph class $\mathcal{G}$ is said to be \textit{clean} if the only induced subgraph obstructions to bounded treewidth in $\mathcal{G}$ are in fact the basic ones, that is, for every $t\in \poi$,  there is a constant $n=n(t)\in \poi$ such that every $t$-clean graph in $\mathcal{G}$ has treewidth less than $n$. An exact analog of Theorem~\ref{wallminor} for induced subgraphs is therefore equivalent to a characterization of all families $\mathcal{H}$ of graphs for which the class of all $\mathcal{H}$-free graphs is clean.

This appears to be out of reach of the current techniques, and even formulating a conjecture seems rather difficult. Nevertheless, the known ``non-basic'' obstructions may provide some insight into the structural properties of a family $\mathcal{H}$ with the above property. For example, from Theorem~\ref{thm:arrayproperties} combined with the definition of a clean class, we deduce that:

\begin{observation}\label{obs:necessary}
    Let $\mathcal{H}$ be a family of graphs such that the class of $\mathcal{H}$-free graphs is clean. Then there exists $n_0\in \poi$ such that for every $n\geq n_0$ and every $n$-array $A$, there is a graph $H\in \mathcal{H}$ which is isomorphic to an induced subgraph of $A$.  
\end{observation}

The converse to Observation~\ref{obs:necessary} is not true in general; indeed, all other non-basic obstructions discovered so far \cite{deathstar,layered-wheels} are counterexamples to this converse:
\begin{itemize}
 \item Let $\mathcal{H}$ be the family of all graphs which are the disjoint unions of two cycles. Then every $n$-array with $n\geq 4$ contains a graph in $\mathcal{H}$, while it is proved in \cite{deathstar} that there are $3$-clean $\mathcal{H}$-free graphs of arbitrarily large treewidth. 
  \item Let $\mathcal{H}$ be the family of all subdivisions of $K_{2,3}$, also known as the \textit{thetas}. Then it is readily observed that every $n$-array with $n\geq 3$ contains a theta, whereas it is proved in \cite{layered-wheels} that the class of theta-free graphs is not clean.
\end{itemize}

It turns out that what allows these counterexamples to occur is the fact that the corresponding family $\mathcal{H}$ of graphs that we are forbidding is infinite. Our main result in this paper shows for a \textit{finite} set $\mathcal{H}$ of graphs, the necessary condition from Observation~\ref{obs:necessary} is in fact sufficient:

\begin{theorem}\label{thm:maincute}
    Let $\mathcal{H}$ be a finite set of graphs. Then the class of all $\mathcal{H}$-free graphs is clean if and only if there are only finitely many $n\in \poi$ for which there is an $\mathcal{H}$-free $n$-array.
\end{theorem}
This may also be regarded as a natural strengthening of the main result from \cite{twvii}, which handles the special case where $\mathcal{H}$ is a singleton:

\begin{theorem}[Abrishami, Alecu, Chudnovsky, Hajebi and Spirkl \cite{twvii}]\label{tw7}
    Let $H$ be a graph. Then the class of all $H$-free graphs is clean if and only if every component of $H$ is a subdivided star, that is, a tree with at most one vertex of degree more than two.
\end{theorem}

\subsection{Tassels and tasselled families}
Note that by Observation~\ref{obs:necessary}, we only need to prove the ``if'' implication of Theorem~\ref{thm:maincute}. To that end, we find it technically most convenient to reformulate the  ``if'' implication in terms of ``the columns of the arrays,'' which we refer to as ``tassels.''

Let us make this precise. A \textit{strand} is a graph $F$ obtained from a path $P$ by adding a new vertex $x$ with at least one neighbor in $P$, and we say $F$ is a \textit{$c$-strand}, for $c\in \poi$, if $x$ is not adjacent to the first and last $c$ vertices of $P$. We call $x$ and $P$ the \textit{neck of $F$} and the \textit{path of $F$}, respectively. For $c\in \poi$, by a \textit{$c$-tassel} we mean a graph $T$ obtained from at least $c$ pairwise disjoint copies of a $c$-strand $F$ by identifying their necks into a single vertex, called the \textit{neck of $T$}. We also refer to each copy of $F$ in $T$ as a \textit{strand of $T$}, and to the path of each copy of $F$  as a \textit{path of $T$} (see Figure~\ref{fig:tassels}).
\begin{figure}[t!]
\centering
\includegraphics[scale=0.6]{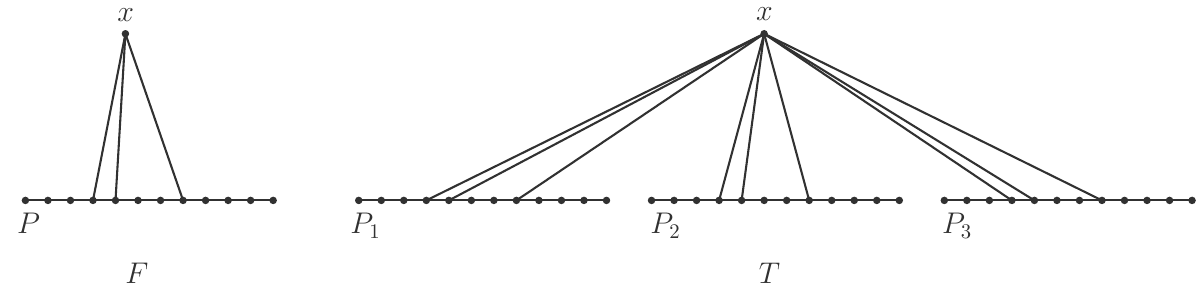}
\caption{Left: a $3$-strand $F$ with neck $x$ and path $P$. Right: a $3$-tassel $T$ with neck $x$ and paths $P_1,P_2,P_3$, obtained from three copies of $F$ as its strands.}
\label{fig:tassels}
\end{figure}

We say that a family $\mathcal{H}$ of graphs is \textit{tasselled} if  there is a constant $c=c(\mathcal{H})\in \poi$ with the property that
for every $c$-tassel $T$
there is  $H \in \mathcal{H}$ such that $T$  contains each component of
$H$. Also, if  $\mathcal{H}$ is finite, we define $||\mathcal{H}||=\sum_{H\in \mathcal{H}}|V(H)|$. It follows that:  

\begin{theorem}\label{thm:onlyif}
      Let $\mathcal{H}$ be a finite set of graphs such that there is no $\mathcal{H}$-free $n$-array except possibly for finitely many $n\in \poi$. Then $\mathcal{H}$ is tasselled.
\end{theorem}
\begin{proof}
By the assumption, there exists $n_0\in \poi$ such that there is no $\mathcal{H}$-free $n$-array for any $n\geq n_0$. Let $c=\max\{n_0,||\mathcal{H}||\}$. In order to prove that $\mathcal{H}$ is tasselled, we show that for every $c$-tassel $T$, there is a graph $H\in \mathcal{H}$ such that $T$ contains each component of $H$.

Let $d\geq c$ be the number of paths of $T$. Construct an $d$-array $A$ as follows. Start with $d$ pairwise disjoint copies $T_1,\ldots, T_d$ of $T$. For each $i\in \poi_d$, let $x_i$ be the neck of $T_i$ and fix an enumeration $P^i_1,\ldots, P^i_{d}$  of the paths of $T_i$. For every $i,j\in \poi_d$, fix a labelling  $u^i_j,v^i_j$ of the ends of $P^i_j$. Then, for every  $i\in \poi_{d-1}$ and every $j\in \poi_{d}$, add an edge between $u^{i}_j$ and $v^{i+1}_j$ (see Figure~\ref{fig:arraywithtassels}).
\begin{figure}[t!]
\centering
\includegraphics[scale=0.7]{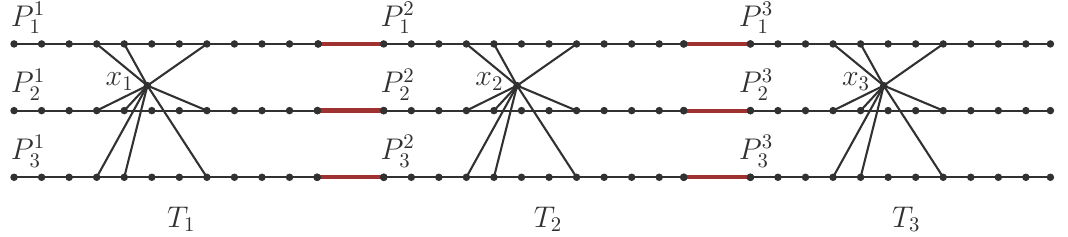}
\caption{The $3$-array $A$ as described in the proof of Theorem~\ref{thm:onlyif}, obtained from three copies $T_1,T_2,T_3$ of the $3$-tassel $T$ in Figure~\ref{fig:tassels}.}
\label{fig:arraywithtassels}
\end{figure}

Since $A$ is an $d$-array with $d\geq c\geq n_0$, it follows that $A$ is not $\mathcal{H}$-free, and so we may choose a graph $H\in \mathcal{H}$ which is contained in $A$.

Let $K$ be a component of $H$. Our goal is to show that $K$ is contained in $T$. This is immediate if $K$ is a path, because $|K|\leq ||\mathcal{H}||\leq c$ and $T$ is a $c$-tassel. Thus, we may assume that $K$ is not a path. Since $H$ is contained in $A$, it follows that some induced subgraph $K'$ of $A$ is isomorphic to $K$. In particular, $K'$ is a connected graph on at most $||\mathcal{H}||$ vertices which is not a path. Also, from the construction of $A$, it is readily observed that the necks $x_1,\ldots,x_d$ of $T_1,\ldots, T_d$ are pairwise at distance at least $2c+3>||\mathcal{H}||\geq |K'|$ in $G$. Consequently, there exists exactly one $i\in \poi_{d}$ for which $x_i$ belongs to $V(K')$.

Now, since $K'$ is connected, it follows that for every component $P$ of $K'\setminus \{x_i\}$, we have $P\subseteq V(A)\setminus \{x_1,\ldots, x_d\}$ and $x_i$ has a neighbor in $P$. This, along with the fact that $P$ is connected and $|P|<||\mathcal{H}||\leq c$, implies that $P\subseteq P^i_j\subseteq T_i\setminus \{x_i\}$ for some $j\in \poi_{d}$. In conclusion, we have shown that $K'\setminus \{x_i\}\subseteq T_i\setminus \{x_i\}$, and so $K'\subseteq T_i$. Hence, $K$ is contained in $T$, as desired.
\end{proof}
 In view of Theorem~\ref{thm:onlyif}, in order to prove Theorem~\ref{thm:maincute}, it suffices to show that:
 \begin{theorem}\label{mainiftassel}
     Let $\mathcal{H}$ be a finite set of graphs which is tasselled. Then the class of $\mathcal{H}$-free graphs is clean.
 \end{theorem}
 The rest of the paper is devoted to the proof of Theorem~\ref{mainiftassel}, which is completed in Section~\ref{sec:end}. Note also that Theorems~\ref{thm:onlyif} and \ref{mainiftassel} combined with Observation~\ref{obs:necessary} imply the following (see also Figure~\ref{fig:equivalence}).
\begin{corollary}\label{cor:equivalence}
    The following are equivalent for every finite set $\mathcal{H}$ of graphs.
    \begin{itemize}
        \item There is no $\mathcal{H}$-free $n$-array except possibly for finitely many $n\in \poi$.
        \item $\mathcal{H}$ is tasselled.
         \item The class of all $\mathcal{H}$-free graphs is clean.
    \end{itemize}
\end{corollary}
\begin{figure}[t!]
    \centering
    \includegraphics[scale=0.8]{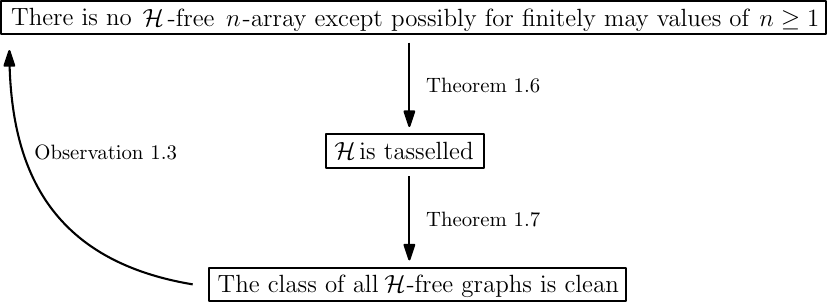}
    \caption{Corollary~\ref{cor:equivalence}}
    \label{fig:equivalence}
\end{figure}

\subsection{Outline}

To prove the Theorem~\ref{mainiftassel}, we proceed as follows. First, in Section \ref{sec:goodtoexcellent}, we show that a finite family is tasselled if and only if it is ``hassled,'' which is described by containment not in a $c$-tassel, but in a less restricted $c$-hassle: while similar to a tassel, the paths are replaced by walks such that any short stretch is an induced path, the adjacency from the neck to the walks need not be the same, and we may have edges between walks.  

Then, our goal is to show that $t$-clean graphs of large treewidth contain $c$-hassles. To do so, we start with some preliminary Ramsey-type results in Section \ref{sec:prilim}. In Section \ref{sec:blocks}, we show that in ``blocks,'' which are sets of vertices with pairwise many paths between them, we can arrange that those paths are long. This is useful for getting many long induced paths, which will form part of our $c$-hassles. 

In Section \ref{sec:model}, we get a structure of large treewidth consisting of one of the following: 
\begin{itemize}
    \item A large set of paths, and a large set of vertices, each with a neighbor in every path; or
    \item Two large sets of paths such that each path in one set has an edge to each path in the other, but no vertex has neighbors in many paths. 
\end{itemize}
In Section \ref{sec:turbinconst}, we show how each of these configurations leads to a $c$-hassle. In the former case, this is almost immediate, whereas the latter case is more involved: we find a block whose vertex set is contained in the union of the paths, and leverage the interaction between the two structures, along with the fact that blocks are long, to find our $c$-hassle.  

In Section \ref{sec:end}, we put everything together and prove our main result. 

\section{Hassled families}\label{sec:goodtoexcellent}

In this section, we introduce the notion of a ``hassled family'' as a
variant of tasselled families, and show that for finite families of graphs, these two properties are in fact equivalent. This in turn reduces Theorem~\ref{mainiftassel} to: \textit{for every finite and \textbf{hassled} family $\mathcal{H}$ of graphs, the class of all $\mathcal{H}$-free graphs is clean.} The remainder of this paper is then occupied with a proof of the latter statement, which will appear as Theorem~\ref{mainif} in Section~\ref{sec:end}. In order to define hassled families, first we need another notion, that of a  ``$c$-hassle,'' which is similar to a $c$-tassel except its paths are replaced by walks that are ``locally'' isomorphic to a path, and, up to sparsity, there may be additional edges between these walks.

Let us define all this formally. Let $n\in \poi$ and let $P$ be an $n$-vertex path (as a graph). Then we write $P = p_1 \dd \cdots \dd p_n$ to mean that $V(P) = \{p_1, \dots, p_n\}$ and $p_i$ is adjacent to $p_j$ if and only if $|i-j| = 1$. We call the vertices $p_1$ and $p_n$ the \emph{ends of $P$}, and refer to $P\setminus \{p_1,p_n\}$ as the \emph{interior of $P$}, denoted $P^*$. For vertices $u,v\in V(P)$, we denote by $u\dd P\dd v$ the subpath of $P$ from $u$ to $v$. Recall that the \textit{length} of a path is the number of edges in it.

Let $G$ be a graph. A {\em path in $G$} is an induced subgraph of $G$ which is a path. For $X,Y\subseteq V(G)$. We say $X$ is \textit{complete to} $Y$ in $G$ if every vertex of $X$ is adjacent to every vertex in $Y$, and we say $X$ and $Y$ are \textit{anticomplete} in $G$ if there is no edge in $G$ with an end in $X$ and an end in $Y$. If $x\in V(G)$, we say $x$ is \textit{complete (anticomplete) to} $Y$ if $\{x\}$ and $Y$ are complete (anticomplete) in $G$.

For $n\in \poi\cup \{0\}$, by an \textit{$n$-segment} we mean a set $S$ of at most $n$ consecutive integers; in particular, $\poi_{n}$ is an $n$-segment. A graph $W$ is a \textit{walk} if there is $n_W\in \poi$ and a surjective map $\varphi_W:\poi_{n_W}\rightarrow V(W)$ such that for every $i\in \poi_{n_W-1}$, we have $\varphi_W(i)\varphi_W(i+1)\in E(W)$ (one may in fact observe that a graph $W$ is a walk if and only if it is connected). Given a walk $W$ along with choices of $n_W$ and $\phi_W$, for $c\in \poi$, we refer to $\varphi_W(\poi_{c})$ and $\varphi_W(\poi_{n_W}\setminus \poi_{n_W-c})$ as the \textit{first $c$ vertices of $W$} and the \textit{last $c$ vertices of $W$}, respectively. We also say  $W$ is \textit{$c$-stretched} if $\varphi_W(S)$ is a path in $W$ for every $c$-segment $S\subseteq \poi_{n_W}$ (see Figure~\ref{fig:Walk} -- intuitively, this means a snake of length $c-1$ traversing through $W$ can never see/hit itself).

We now define a \textit{$c$-hassle}, where $c\in \poi$, to be a graph $\Xi$ obtained from at least $c$ pairwise disjoint $c$-stretched walks, called the \textit{walks of $\Xi$}, by adding edges arbitrarily between the walks, and then adding a vertex $x$, called the \textit{neck of $\Xi$}, which has a neighbor in each walk and which is anticomplete to the first and last $c$ vertices of each walk (see Figure~\ref{fig:hassle}).

\begin{figure}[t!]
\centering
\includegraphics[scale=0.6]{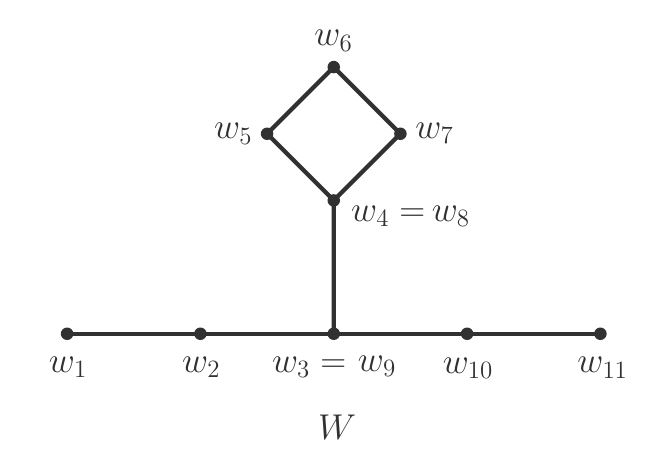}
\caption{A $3$-stretched walk $W$ with $n_W=11$ and $\varphi_W(i)=w_i$ for all $i\in \poi_{11}$.}
\label{fig:Walk}
\end{figure}

For a family $\mathcal{H}$ of graphs, we say $\mathcal{H}$ is \textit{hassled} if for every $t\in \poi$, there is $c=c(\mathcal{H},t)\in \poi$ with the property that for every $(K_{t+1},K_{t,t})$-free $c$-hassle $T$ 
there exists $H \in \mathcal{H}$ such that $T$ contains every component of $H$.
In particular, observe that every $c$-tassel is a $(K_4,K_{3,3})$-free $c$-hassle, and so every hassled family is tasselled. More importantly, for finite families, the converse is also true: 

\begin{theorem}\label{thm:tassellediffhassled}
Let $\mathcal{H}$ be a finite set of graphs. Then $\mathcal{H}$ is tasselled if and only if $\mathcal{H}$ is hassled. 
\end{theorem}

\begin{figure}[t!]
\centering
\includegraphics[scale=0.6]{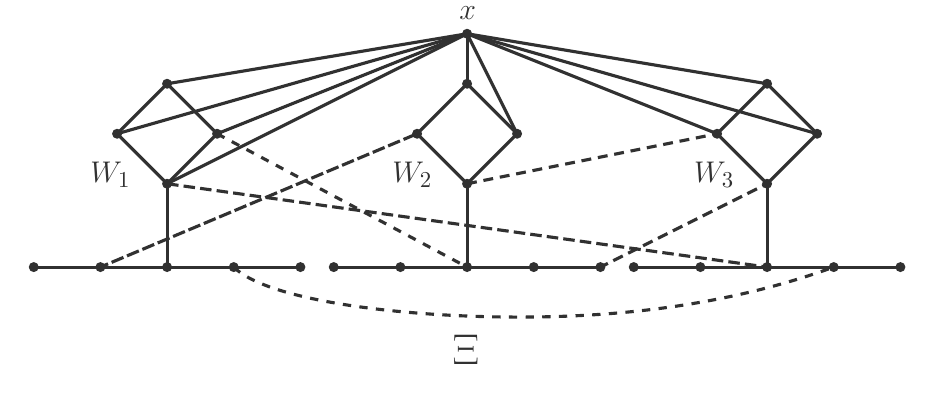}
\caption{A $3$-hassle $\Xi$ with neck $x$, where the walks $W_1,W_2,W_3$ of $\Xi$ are three copies of the $3$-stretched walk $W$ from Figure~\ref{fig:Walk}.}
\label{fig:hassle}
\end{figure}

 The goal of this section is to prove Theorem~\ref{thm:tassellediffhassled}, beginning with two lemmas:
 \begin{lemma}\label{lem:multitassel}
   Let $b,c,k\in \poi$, let $\mathcal{T}$ be a family of $c$-tassels, each with exactly $c$ paths, such that $|\mathcal{T}|\geq bkc^{k}$. For each $T\in \mathcal{T}$, let $x_T$ be the neck of $T$ and fix an enumeration $P^T_1,\ldots, P^T_c$ of the paths of $T$. Let $K$ be a connected graph on at most $k$ vertices which is not a path, and assume that for every $T\in \mathcal{T}$, there is an isomorphism $f_{T}$ from $K$ to an induced subgraph of $T$; in particular, $T$ contains $K$. Then there exist $x'\in V(K)$ and $\mathcal{T}'\subseteq \mathcal{T}$ with $|\mathcal{T}'|=b$ for which the following hold.
\begin{enumerate}[\rm (a)]
        \item\label{lem:multitassel_a} For every $T\in \mathcal{T}'$, we have $f_T(x')=x_T$.
        \item\label{lem:multitassel_b} For every component $L$ of $K\setminus \{x'\}$, there exists $i(L)\in \{1,\ldots, c\}$ such that for every $T\in \mathcal{T}'$, we have $f_{T}(L)\subseteq P^T_{i(L)}$.
        \end{enumerate}
 \end{lemma}
 \begin{proof}
For each $T\in \mathcal{T}$, since $K$ is not a path and $f_T$ is an isomorphism, it follows that there is a unique vertex $x'_T\in V(K)$ such that $f_T(x'_T)=x_T$. Also, since $K$ has at most $k$ vertices, it follows that:

\sta{\label{st:sameneck} There exist $x'\in  V(K)$ and $\mathcal{T}_1\subseteq \mathcal{T}$ with $|\mathcal{T}_1|=bc^k$ such that for every $T\in \mathcal{T}_1$, we have $f_T(x')=x_T$.} 

By \eqref{st:sameneck}, for each $T\in \mathcal{T}_1$ and every component $L$ of $K\setminus \{x'\}$, we have $f_{T}(L)\subseteq T\setminus \{x_T\}$, which in turn implies that there exists $i(L,T)\in \poi_{c}$ for which we have $f_{T}(L)\subseteq P^T_{i(L,T)}$. Now, since $K\setminus \{x'\}$ has at most $k$ components and since $|\mathcal{T}_1|\geq bc^k$, it follows that there exists $\mathcal{T}'\subseteq \mathcal{T}_1\subseteq \mathcal{T}$ with $|\mathcal{T'}|=b$, as well as $i(L)\in \poi_{c}$ for every component $L$ of $K\setminus \{x'\}$, such that for each $T\in \mathcal{T}'$, we have $i(L,T)=i(L)$. Hence, $x'$ and $\mathcal{T}'$ satisfy \ref{lem:multitassel}\ref{lem:multitassel_b}. Moreover, from \eqref{st:sameneck}, it follows that $x'$ and $\mathcal{T}'$ satisfy \ref{lem:multitassel}\ref{lem:multitassel_a}, as desired.
\end{proof}

\begin{lemma}\label{lem:multihassle}
    For every finite and tasselled set $\mathcal{H}$ of graphs,  there is a constant $\xi=\xi(\mathcal{H})\in \poi$ with the following property.  For every $\xi$-hassle $\Xi$ with neck $x$, there is a graph $H\in \mathcal{H}$ such that for every component $K$ of $H$, one of the following holds.
    \begin{enumerate}[\rm (a)]
        \item\label{lem:multihassle_a} $K$ is a path.
        \item\label{lem:multihassle_b} There exists $x'\in V(K)$ and a map $f:V(K)\rightarrow V(\Xi)$ with $f^{-1}(\{x\})=\{x'\}$, such that for every component $L$ of $K\setminus \{x'\}$, the restriction of $f$ to 
$\{x'\}\cup L$ is an isomorphism from $K[\{x'\}\cup L]$ to $\Xi[\{x\}\cup f(L)]$.
    \end{enumerate}
\end{lemma}
\begin{proof}
  Since $\mathcal{H}$ is tasselled, it follows that  there is a constant $c=c(\mathcal{H})\in \poi$ with the property that for every $c$-tassel  $T$
there is  $H \in \mathcal{H}$   such that $T$  contains every component of
$H$.
We claim that
$$\xi=\xi(\mathcal{H})=||\mathcal{H}||^2c^{||\mathcal{H}||+1}$$
satisfies the lemma.
To see this, let $\Xi$ be a $\xi$-hassle with neck $x$, and let $\mathcal{W}$ be the collection of all walks of $\Xi$. Recall that each walk in $\Xi$ is $\xi$-stretched.

For each $W\in \mathcal{W}$, let $n_W$ and $\varphi_W$ be as in the definition of a walk, and construct a $c$-tassel $T_W$ as follows. Let $P^W_1,\ldots, P^W_{c}$ be $c$ pairwise disjoint and anticomplete paths, each on $n_W$ vertices, and for every $i\in \poi_c$, choose a bijection $g^W_i:V(P^W_i)\rightarrow \poi_{n_W}$ such that $p,p'\in V(P^W_i)$ are adjacent in $P^W_i$ if and only if $|g^W_i(p)-g^W_i(p')|=1$ (note that there are only two such bijections). Let $T_W$ be the graph obtained from $P^W_1,\ldots, P^W_{c}$ by adding a vertex $x_W$ such that for every $i\in \poi_{c}$ and every $p\in V(P^W_i)$, the vertex $x_W$ is adjacent to $p$ in $T_W$ if and only if $x$ is adjacent to $\varphi_W(g^W_{i}(p))$ in $\Xi$ (see Figure~\ref{fig:walktopath}).

\begin{figure}[t!]
\centering
\includegraphics[scale=0.6]{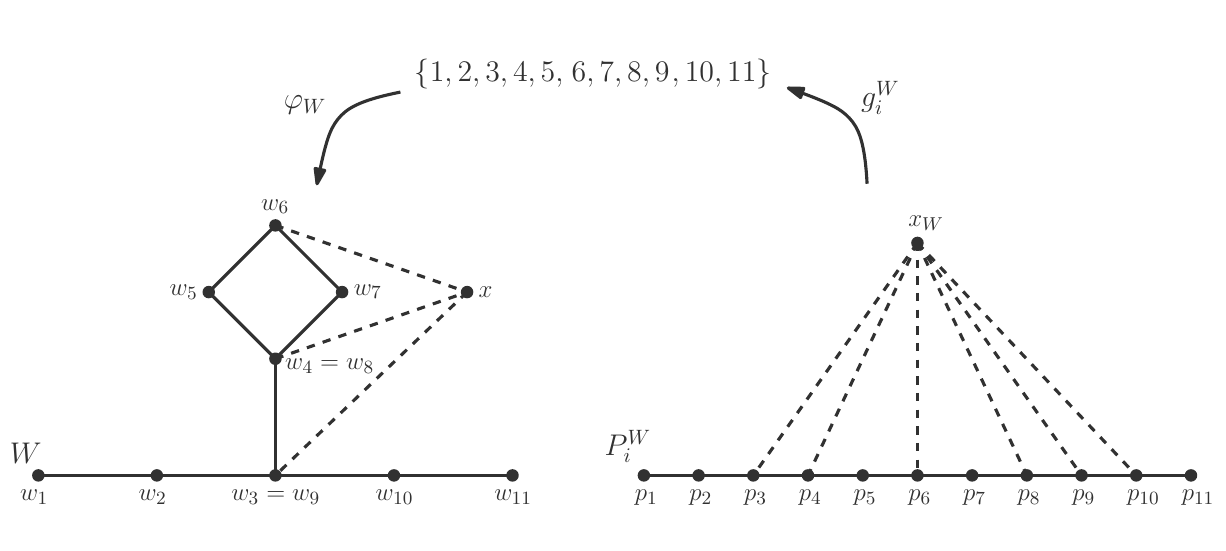}
\caption{Left: the walk $W$ with $n_W=11$ and $\varphi_W(j)=w_j$ for all $i\in \poi_{11}$. Right: the path $P_i^W$ with $g^W_i(p_j)=j$ for all $j\in \poi_{11}$. Observe that for every $j\in \poi_{11}$, we have $\varphi_W(g^W_{i}(p_j))=w_j$, and $x_W$ is adjacent to $p_j$ in $T_W$ if and only if $x$ is adjacent to $w_j$ in $\Xi$.}
\label{fig:walktopath}
\end{figure}

Note that for every $W\in \mathcal{W}$, in the graph $\Xi$, the vertex $x$ has a neighbor in $W$ and no neighbor among the first and last $\xi$ vertices of $W$. In particular, from the construction and the fact that $\xi\geq c$, it follows that for every $i\in \poi_{c}$, $x$ has a neighbor in $P^W_i$ and no neighbor among the first and last $c$ vertices of $P^W
_i$. Thus, for every $W\in \mathcal{W}$, the graph $T_W$ is a $c$-tassel with neck $x_W$ and paths $P^W_1,\ldots, P^W_{c}$.

Also, since $\mathcal{H}$ is tasselled, it follows from the choice of $c$ that for every $W\in \mathcal{W}$ there exists $H \in \mathcal{H}$ such that
the $c$-tassel $T_W$ contains every  component of $H$.
Consequently, since $|\mathcal{W}|\geq \xi$ and $|\mathcal{H}|\leq ||\mathcal{H}||$, it follows that there exists $H\in \mathcal{H}$ and $\mathcal{W}'\subseteq \mathcal{W}$ with $|\mathcal{W}'|=||\mathcal{H}||c^{||\mathcal{H}||+1}$ such that for every $W\in \mathcal{W}'$, the $c$-tassel $T_W$ contains every component of $H$.

We now prove that $H$ satisfies \ref{lem:multihassle}. Let $K$ be a component of $H$ which is not a path. We wish to show that $K$ satisfies \ref{lem:multihassle}\ref{lem:multihassle_b}. Note that for every $W\in \mathcal{W}'$, the $c$-tassel $T_W$ contains $K$, and so there is an isomorphism $f_{W}$ from $K$ to an induced subgraph of $T_W$. This allows for an application of Lemma~\ref{lem:multitassel} to $\mathcal{T}=\{T_W:W\in \mathcal{W}'\}$ and $K$. Since $|V(K)|\leq |V(H)|\leq ||\mathcal{H}||$, we deduce that there exists $x'\in V(K)$ as well as $W_1,\ldots, W_c\in \mathcal{W}'$, such that $x'$ and $\mathcal{T}'=\{T_{W_1},\ldots, T_{W_c}\}$ satisfy Lemma~\ref{lem:multitassel}\ref{lem:multitassel_a} and \ref{lem:multitassel_b}.

Henceforth, for each $j\in \poi_{c}$, we write 
$$T_j=T_{W_j};  \quad x_j=x_{W_j}; \quad f_j=f_{W_j}; \quad n_j=n_{W_j}; \quad \varphi_j=\varphi_{W_j}.$$
Also, for $i,j\in \poi_{c}$, we write 
$$P_{i,j}=P^{W_j}_i; \quad g_{i,j}=g^{W_j}_i.$$
The outcomes of Lemma~\ref{lem:multitassel} can now be rewritten as:

\sta{\label{st:lemtasselagain} The following hold.
\begin{itemize}
        \item For every $j\in \poi_{c}$, we have $f_j(x')=x_{j}$.
        \item For every component $L$ of $K\setminus \{x'\}$, there exists $i(L)\in \{1,\ldots, c\}$ such that for every $j\in \poi_{c}$, we have $f_{j}(L)\subseteq P_{i(L),j}$.
        \end{itemize}}
        
On the other hand, since $|K|\leq |H|\leq ||\mathcal{H}||\leq \xi$, it follows that every component of $K\setminus \{x'\}$ is a path on less than $\xi$ vertices. This, combined with the second bullet of \eqref{st:lemtasselagain}, yields the following:

\sta{\label{st:getsegment} For every $j\in \poi_{c}$ and every component $L$ of $K\setminus \{x'\}$, there exists a $\xi$-segment $S_{j,L}\subseteq \poi_{n_j}$ for which we have 
$f_j(L)=g_{i(L),j}^{-1}(S_{j,L})\subseteq P_{i(L),j}$.}

Let us now finish the proof. Define a map $f:V(K)\rightarrow V(\Xi)$ as follows. Let $f(x')=x$, and for every component $L$ of $K\setminus \{x'\}$ and every $y\in L$, let
$$f(y)=\varphi_{i(L)}(g_{i(L),i(L)}(f_{i(L)}(y))).$$

\begin{figure}[t!]
\centering
\includegraphics[scale=0.6]{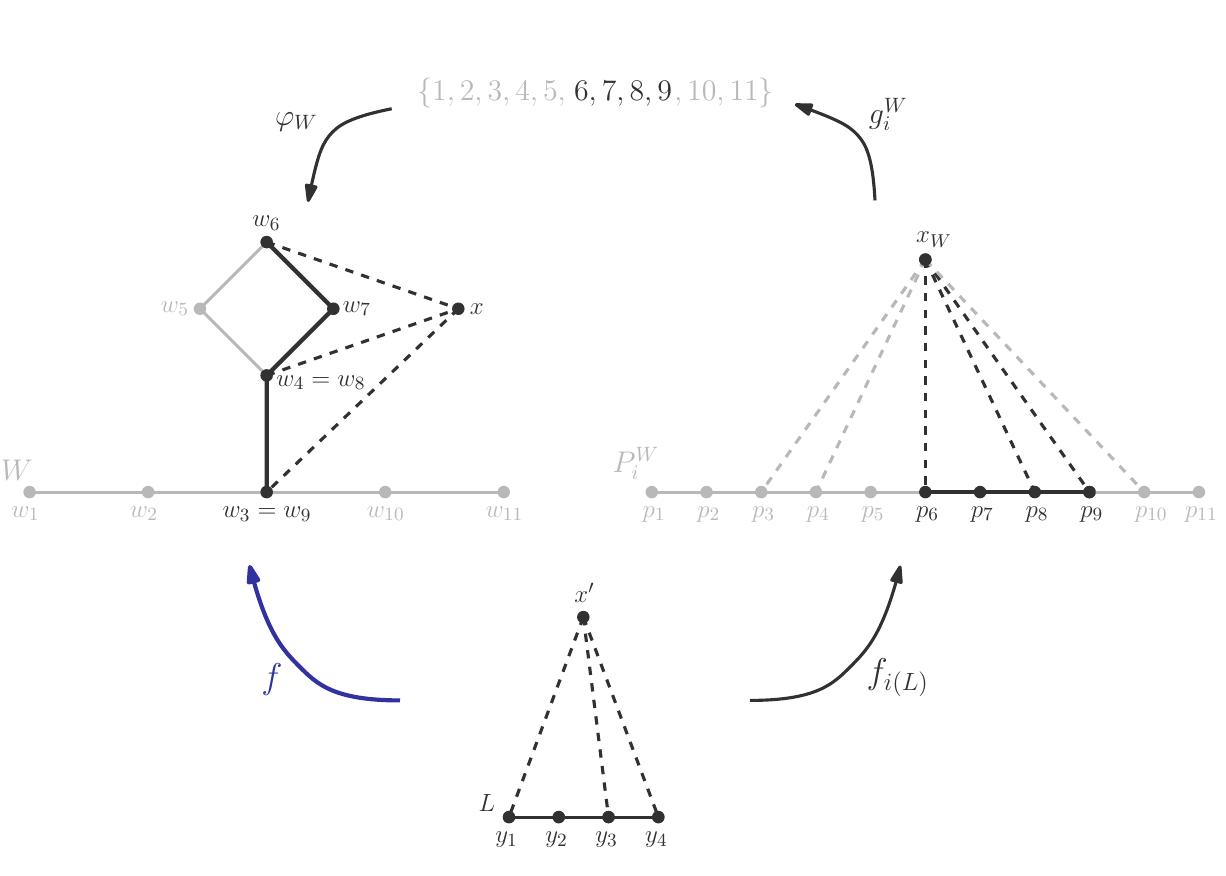}
\caption{The map $f:V(K)\rightarrow V(\Xi)$. For each $j\in \poi_{4}$, we have $f_{i(L)}(y_j)=p_{j+5}$, which yields $f(y_j)=w_{j+5}$.}
\label{fig:walkramsey}
\end{figure}

See Figure~\ref{fig:walkramsey}. We prove that $f$ satisfies Lemma~\ref{lem:multihassle}\ref{lem:multihassle_b}. Let $L$ be a component of $K\setminus \{x'\}$. By \eqref{st:getsegment}, we have $f(L)=\varphi_{i(L)}(S_{i(L),L})\subseteq W_{i(L)}\subseteq V(\Xi)\setminus \{x\}$, and so we have  $f^{-1}(\{x\})=\{x'\}$. This, along with the assumption that $W_{i(L)}$ is a $\xi$-stretched walk, implies that $
f(L)$ is a path in $W_{i(L)}$. In particular, the restriction of $f$ to $L$ is an isomorphism from $K[L]$ to $\Xi[f(L)]$.

It remains to show that for every $y\in L$, the vertices $x',y$ are adjacent in $K$ if and only if $x,f(y)$ are adjacent in $\Xi$. To that end, note that since $f_{i(L)}$ is an isomorphism from $K$ to an induced subgraph of $T_{i(L)}$, it follows from the first bullet of \eqref{st:lemtasselagain}  that $x'$ is adjacent to $y$ in $K$ if and only if $f_{i(L)}(x')=x_{i(L)}$ is adjacent to $f_{i(L)}(y)\in P_{i(L),i(L)}$ in $T_{i(L)}$. In addition, from the definition of $T_i$, it follows that $x_{i(L)}$ is adjacent to $f_{i(L)}(y)\in P_{i(L),i(L)}$ in $T_{i(L)}$ if and only if  $x$ is adjacent to $\varphi_{i(L)}(g_{i(L),i(L)}(f_{i(L)}(y)))=f(y)$ in $\Xi$.
This completes the proof of Lemma~\ref{lem:multihassle}.
\end{proof}

We also need the following well-known result.

\begin{lemma}[See Lemma 2 in \cite{lozin}]\label{ramsey2}
For all $q,r,t\in \poi$  there is a constant $\Delta=\Delta(q,r,t)\in \poi$ with the following property.  Let $G$ be a $(K_{t+1},K_{t,t})$-free graph. Let $\mathcal{X}$ be a collection of pairwise disjoint subsets of $V(G)$, each of cardinality at most $r$, with $|\mathcal{X}|\geq \Delta$. Then there are $q$ distinct sets $X_1,\ldots, X_q\in \mathcal{X}$ which are pairwise anticomplete in $G$.
\end{lemma}

We can now prove Theorem~\ref{thm:tassellediffhassled}, which we restate:

\setcounter{theorem}{0}
\begin{theorem}
Let $\mathcal{H}$ be a finite set of graphs. Then $\mathcal{H}$ is tasselled if and only if $\mathcal{H}$ is hassled. 
\end{theorem}
\begin{proof}
    The ``if'' implication is clear as discussed at the beginning of this section. To prove the ``only if'' implication, assume that $\mathcal{H}$ is a finite set of graphs which is tasselled. Let $\xi=\xi(\mathcal{H})\in \poi$ be as in Lemma~\ref{lem:multihassle}. For every $t\in \poi$, let $\Delta=\Delta(||\mathcal{H}||,||\mathcal{H}||,t)$ be as in Lemma~\ref{ramsey2}, and let $$c=c(\mathcal{H},t)=\xi\Delta||\mathcal{H}||^2.$$
    We prove that for every $(K_{t+1},K_{t,t})$-free $c$-hassle $\Xi$, there exists a graph $H\in \mathcal{H}$ such that $\Xi$ contains each component of $H$. This will show that $\mathcal{H}$ is hassled.

Let $x$ be the neck of $\Xi$. By the choice of $c$, we may choose a set $\mf{W}$ of $\Delta||\mathcal{H}||^2$ pairwise disjoint families of walks of $\Xi$, each of cardinality $\xi$. It follows that for every $\mathcal{W}\in \mf{W}$, the graph $\Xi_{\mathcal{W}}=\Xi[\{x\}\cup (\bigcup_{W\in \mathcal{W}}V(W))]$ is a $\xi$-hassle with neck $x$, and with $\mathcal{W}$ as its set of walks.

Since $\mathcal{H}$ is tasselled, and by the choice of $\xi$, for each $\mathcal{W}\in \mf{W}$, we can apply Lemma~\ref{lem:multihassle} to $\mathcal{H}$ and $\Xi_{\mathcal{W}}$, and obtain a graph $H_{\mathcal{W}}\in \mathcal{H}$ satisfying Lemma~\ref{lem:multihassle}\ref{lem:multihassle_a} and \ref{lem:multihassle_b}. Moreover, since $|\mathcal{H}|\leq ||\mathcal{H}||$ and $|\mf{W}|=\Delta||\mathcal{H}||^2$, it follows that there exist $H\in \mathcal{H}$ and $\mf{X}\subseteq \mf{W}$ with $|\mf{X}|=\Delta||\mathcal{H}||$ such that for every $\mathcal{W}\in \mf{X}$, we have $H_{\mathcal{W}}=H$. More explicitly, for every component $K$ of $H$, one of the following holds.
\begin{itemize}
    \item $K$ is a path. 
    \item For every $\mathcal{W}\in \mf{X}$, there exists $x'_{\mathcal{W}}\in V(K)$ and an injective map $f_{\mathcal{W}}:V(K)\rightarrow V(\Xi_{\mathcal{W}})$ with $f^{-1}(\{x\})=\{x'_{\mathcal{W}}\}$, such that for every component $L$ of $K\setminus \{x'_{\mathcal{W}}\}$, the restriction of $f_{\mathcal{W}}$ to $\{x'_{\mathcal{W}}\}\cup L$ is an isomorphism from $K[\{x'_{\mathcal{W}}\}\cup L]$ to $\Xi[\{x\}\cup f_{\mathcal{W}}(L)]$.
\end{itemize}

To conclude the proof, it suffices to show that $\Xi$ contains every component $K$ of  $H$. Assume that $K$ is a path. Since $\Xi$ is a $c$-hassle, it follows that $\Xi$ contains a path on $c$ vertices. But now we are done because $|K|\leq ||\mathcal{H}||\leq c$. Consequently, we may assume that $K$ is not a path, and so the second bullet above holds for $K$ and every $\mathcal{W}\in \mf{X}$. In addition, from $|K|\leq ||\mathcal{H}||$ and $|\mf{X}|=\Delta||\mathcal{H}||$, it follows that there exist $x'\in V(K)$ and $\mf{Y}\subseteq \mf{X}$ with $|\mf{Y}|=\Delta$ such that for every $\mathcal{W}\in \mf{Y}$, we have $x'_{\mathcal{W}}=x'$. 

 On the other hand, $\{f_{\mathcal{W}}(K\setminus \{x'\}):\mathcal{W}\in \mf{Y}\}$ is a collection of $\Delta$ pairwise disjoint subsets of $\Xi\setminus \{x\}$, each of cardinality less than $|K|\leq ||\mathcal{H}||$. This, along with the choice of $\Delta$ and the assumption that $\Xi$ is $(K_{t+1},K_{t,t})$-free, allows for an application of Lemma~\ref{ramsey2}. We deduce that:

\sta{\label{st:anticomphassle}There exist $\mathcal{W}_1, \ldots, \mathcal{W}_{||\mathcal{H}||}\in \mf{Y}$ for which the sets $f_{\mathcal{W}_1}(K\setminus \{x'\}),\ldots, f_{\mathcal{W}_{||\mathcal{H}||}}(K\setminus \{x'\})\subseteq \Xi\setminus \{x\}$ are pairwise anticomplete in $\Xi$.}

Now, let $L_1,\ldots, L_k$ be an enumeration of the components of $K\setminus \{x'\}$; then we have $k<|K|\leq ||\mathcal{H}||$. Let  $\mathcal{W}_1, \ldots, \mathcal{W}_{k}\in \mf{Y}$ be as in \eqref{st:anticomphassle}. By the second bullet above, for each $i\in \poi_k$, the restriction of $f_{\mathcal{W}_i}$ to $\{x'\}\cup L_i$ is an isomorphism from $K[\{x'\}\cup L_i]$ and $\Xi[\{x\}\cup f_{\mathcal{W}_i}(L_i)]$ with $f_{\mathcal{W}_i}^{-1}(\{x\})=\{x'\}$. Moreover, by \eqref{st:anticomphassle}, the sets $\{f_{\mathcal{W}_i}(L_i):i\in \poi_{k}\}$ are pairwise disjoint and anticomplete in $\Xi$. Hence, $K$ is isomorphic to 
$$\Xi\left[\{x\}\cup \left(\bigcup_{i=1}^kf_{\mathcal{W}_i}(L_i)\right)\right].$$
This completes the proof of Theorem~\ref{thm:tassellediffhassled}. 
 \end{proof}
 \setcounter{theorem}{5}
\section{A lemma about pairs of sets of vertices}\label{sec:prilim}

We now begin to make our way to the proof of Theorem~\ref{mainif}. In particular, here we prove a Ramsey-type result that we will use several times later. We need the following product version of Ramsey's Theorem.

\begin{theorem}[Graham, Rothschild and Spencer \cite{productramsey}]\label{productramsey}
For all $n,q,r\in \poi$,  there is a constant $\nu(n,q,r)\in \poi$ with the following property. Let $U_1,\ldots, U_n$ be $n$ sets, each of cardinality at least $\nu(n,q,r)$ and let $W$ be a non-empty set of cardinality at most $r$. Let $\Phi$ be a map from the Cartesian product $U_1\times \cdots \times U_n$ to $W$. Then there exist $i\in W$ and $Z_j\subseteq U_j$ with $|Z_j|=q$ for each $j\in \poi_{n}$, such that for every $z\in Z_1\times \cdots\times Z_n$, we have $\Phi(z)=i$.
\end{theorem}

For an induced subgraph $H$ of a graph $G$ and a vertex $x\in V(G)$, we denote by $N_H(x)$ the set of all neighbors of $x$ in $H$, and write $N_H[x]=N_H(x)\cup \{x\}$.  Let $U$ be a set and let $a\in \poi\cup \{0\}$. An \textit{$a$-pair over $U$} is a pair $(A,B)$ of subsets of $U$ with $|A|\leq a$. Two $a$-pairs $(A,B), (A',B')$ are said to be \textit{disjoint} if $B\cap B'=\emptyset$.

\begin{lemma}\label{lem:coolramsey}
For all $a,b\in \poi\cup \{0\}$ and $m\in \poi$,  there is a constant $\Upsilon=\Upsilon(a,b,m)\in \poi$ with the following property. Let $G$ be a graph. Let $\mathcal{B}_1,\ldots, \mathcal{B}_m$ be collections of pairwise disjoint $a$-pairs over $V(G)$, each of cardinality at least $\Upsilon$. Then for every $i\in \poi_{m}$, there exists $\mathcal{B}'_i\subseteq \mathcal{B}_i$ with $|\mathcal{B}'_i|\geq b$ such that for all distinct $i,j\in \poi_{m}$, the following hold.
\begin{enumerate}[\rm (a)]
        \item\label{lem:coolramsey_a} 
        We have $A_i\cap B_j=\emptyset$ for all $(A_i,B_i)\in \mathcal{B}'_{i}$ and $(A_j,B_j)\in \mathcal{B}'_{j}$.

  \item\label{lem:coolramsey_b} Either $A_i$ is anticomplete to $B_j$ in $G$ for all $(A_i,B_i)\in \mathcal{B}'_{i}$ and $(A_j,B_j)\in \mathcal{B}'_{j}$, or for every $(A_i,B_i)\in \mathcal{B}'_{i}$, there exists $x_i\in A_i$ such that $x_i$ has a neighbor in $B_j$ for every $(A_j,B_j)\in \mathcal{B}'_{j}$.
    \end{enumerate}
\end{lemma}
\begin{proof}
    We prove that   $$\Upsilon(a,b,m)=\nu(m,\max\{b,2\},2^{2am^2});$$
    satisfies the theorem, where $\nu(\cdot,\cdot,\cdot)$ is as in Theorem~\ref{productramsey}. Let $\mathcal{B}=\mathcal{B}_1\cup \cdots \cup \mathcal{B}_{m}$.
    For every $a$-pair $(A,B)\in \mathcal{B}$, fix an enumeration $A=\{x^i_A: i\in \poi_{|A|}\}$ of the elements of $A$; recall that $|A|\leq a$.
    For every two $a$-pairs $(A,B),(A',B')\in \mathcal{B}$, let 
    $$I_1(A,B')=\{i\in \poi_{|A|}: x_A^i\in  B'\}\subseteq \poi_{a};$$
 $$I_2(A,B')=\{i\in \poi_{|A|}: N_G(x_A^i)\cap B'\neq \emptyset\}\subseteq \poi_{a}.$$
    
Let $\mathds{M}$ be the set of all $m$-by-$m$ matrices whose entries are subsets of $\poi_{a}$; so we have $|\mathds{M}|=2^{am^2}$. Consider the product $\Pi=\mathcal{B}_1\times \cdots \times \mathcal{B}_m$. For every $z=((A_1,B_1), \cdots, (A_m, B_m))\in \Pi$, define $M_1(z),M_2(z)\in \mathds{M}$ such that for all $i,j\in \poi_{m}$, we have
$$[M_1(z)]_{ij}=I_1(A_i,B_j);$$
$$[M_2(z)]_{ij}=I_2(A_i,B_j).$$
It follows that for every $z\in \Pi$, $M_1(z), M_2(z)$ are unique, and so the map
$\Phi:\Pi\rightarrow \mathds{M}^2$ with $\Phi(z)=(M_1(z),M_2(z))$ is well-defined. This, along with the choice of $\Upsilon$ and Theorem~\ref{productramsey}, implies that there exists $\mathcal{B}'_i\subseteq \mathcal{B}_i$ with $|\mathcal{B}_i|\geq \max\{b,2\}\geq 2$ for each $i\in \poi_{m}$, as well as $M_1,M_2\in \mathds{M}$, such that for every $z\in \mathcal{B}'_1\times \cdots \times \mathcal{B}'_m$, we have $M_1(z)=M_1$ and $M_2(z)=M_2$. Moreover, we deduce:

\sta{\label{st:disjointramsey}  Let $i,j\in \poi_{m}$ be distinct. Then we have $A_i\cap B_j=\emptyset$ for all $(A_i,B_i)\in \mathcal{B}'_{i}$ and $(A_j,B_j)\in \mathcal{B}'_{j}$}

Suppose for a contradiction that there are distinct $i,j\in \poi_{m}$ such that for some $(A_i,B_i)\in \mathcal{B}'_i$ and $(A_j,B_j)\in \mathcal{B}'_j$, we have $A_i\cap B_j\neq \emptyset$. Then we have $I_1(A_i,B_j)\neq \emptyset$. Also, since $|\mathcal{B}'_j|\geq\max\{b,2\}\geq 2$, we may choose $(A'_j,B'_j)\in \mathcal{B}'_j\setminus \{(A_j,B_j)\}$. It follows that $I_1(A_i,B_j)=[M_1]_{ij}=I_1(A_i,B'_j)$ is non-empty. But then $B_j\cap B'_j\neq \emptyset$, a contradiction with the assumption that $(A_j,B_j),(A'_j,B'_j)\in \mathcal{B}'_j\subseteq \mathcal{B}_j$ are disjoint. This proves \eqref{st:disjointramsey}.

\sta{\label{st:antiornot} Let $i,j\in \poi_{m}$ be distinct. Then either $A_i$ is anticomplete to $B_j$ in $G$ for all $(A_i,B_i)\in \mathcal{B}'_{i}$ and $(A_j,B_j)\in \mathcal{B}'_{j}$, or for every $(A_i,B_i)\in \mathcal{B}'_{i}$, there exists $x_i\in A_i$ such that $x_i$ has a neighbor in $B_j$ for every $(A_j,B_j)\in \mathcal{B}'_{j}$.}

Note that for all $(A_i,B_i)\in \mathcal{B}'_{i}$ and $(A_j,B_j)\in \mathcal{B}'_{j}$, we have $I_2(A_i,B_j)=[M_2]_{ij}\subseteq \poi_{a}$. If $[M_2]_{ij}=\emptyset$, then $A_i$ is anticomplete to $B_j$ in $G$ for all $(A_i,B_i)\in \mathcal{B}'_{i}$ and $(A_j,B_j)\in \mathcal{B}'_{j}$. Otherwise, one may choose $k\in [M_2]_{ij}$, and so for each $(A_i,B_i)\in \mathcal{B}'_{i}$, the vertex $x_i=x^k_{A_i}\in A_i$ has a neighbor in $B_j$ for every $(A_j,B_j)\in \mathcal{B}'_{j}$. This proves \eqref{st:antiornot}.

\medskip

Now the result follows from \eqref{st:disjointramsey} and \eqref{st:antiornot}. This completes the proof of Lemma~\ref{lem:coolramsey}.
\end{proof}

\section{Blocks}\label{sec:blocks}
This section collects several results from the literature about ``blocks'' in $t$-clean graphs of large treewidth. We begin with a couple of definitions. Given a set $X$ and $q\in \poi\cup \{0\}$, we denote 
by $2^X$ the power set of $X$ and by $\binom{X}{q}$ the set of all $q$-subsets of $X$. Let $G$ be a graph.
For a collection $\mathcal{P}$ of paths in $G$, we adopt the notation $V(\mathcal{P})=\bigcup_{P\in \mathcal{P}}V(P)$ and $\mathcal{P}^*=\bigcup_{P\in \mathcal{P}}P^*$.
Let $k\in \poi$. A \textit{$k$-block} in $G$ is a set $B$ of at least $k$ vertices in $G$ such that for every $2$-subset $\{x,y\}$ of $B$, there exists a collection $\mathcal{P}_{\{x,y\}}$ of $k$ pairwise internally disjoint paths in $G$ from $x$ to $y$. In addition, we say $B$ is \textit{strong} if the collections $\{\mathcal{P}_{\{x,y\}}:\{x,y\}\subseteq B\}$ can be chosen such that for all distinct $2$-subsets $\{x,y\}, \{x',y'\}$ of $B$, we have $V(\mathcal{P}^*_{\{x,y\}})\cap V(\mathcal{P}_{\{x',y'\}})=\emptyset$. In \cite{twvii}, with Abrishami we proved the following:

\begin{theorem}[Abrishami, Alecu, Chudnovsky, Hajebi and Spirkl \cite{twvii}]\label{thm:tw7block}
  For all $k,t\in \poi$,  there is a constant $\kappa=\kappa(k,t)\in \poi$ such that every $t$-clean graph of treewidth more than $\kappa$ contains a strong $k$-block.
\end{theorem}

Given a graph $G$ and $d,k\in \poi$, we say a (strong) $k$-block $B$ in $G$ is \textit{$d$-short} if for every $2$-subset $\{x,y\}\subseteq B$ of $G$, every path $P\in \mathcal{P}_{\{x,y\}}$ is of length at most $d$. The following shows that large enough short blocks contain strong (and short) ``sub-blocks.''

\begin{theorem}\label{thm:short}
For all $d,k\in \poi$,  there is a constant $\upsilon=\upsilon(d,k)$ with the following property. Let $G$ be a graph and let $B$ be a $d$-short $\upsilon$-block in $G$. Then every $k$-subset $B'$ of  $B$ is a $d$-short strong $k$-block in $G$.
\end{theorem}
\begin{proof}
Let $\upsilon=\upsilon(d,k)=\Upsilon(d+1,k,\binom{k}{2})$, where $\Upsilon(\cdot,\cdot,\cdot)$ is as in Lemma~\ref{lem:coolramsey}. Let $B_0$ be a $d$-short $\upsilon$-block in $G$, and let $B\subseteq B_0$ with $|B|=k$. It follows that for every $2$-subset $\{x,y\}$ of $B$, there is a collection $\mathcal{P}_{\{x,y\}}$ of $\upsilon$ pairwise internally disjoint paths in $G$, each of length at most $d$. Let $\mathcal{B}_{\{x,y\}}=\{(P,P^*):P\in \mathcal{P}_{\{x,y\}}\}$. Then $\mathcal{B}_{\{x,y\}}$ is a collection of $\upsilon$ pairwise disjoint $(d+1)$-pairs over $V(G)$. Thus, the choice of $\upsilon$ allows for an application of Lemma~\ref{lem:coolramsey} to the collections $\{\mathcal{B}_{\{x,y\}}:\{x,y\}\in \binom{B}{2}\}$. We deduce that for every $\{x,y\}\in \binom{B}{2}$, there exists $\mathcal{Q}_{\{x,y\}}\subseteq \mathcal{P}_{\{x,y\}}$ with $|\mathcal{Q}_{\{x,y\}}|\geq k$, such that for all distinct $\{x,y\},\{x',y'\}\in \binom{B}{2}$, the collections $\mathcal{B}'_{\{x,y\}}=\{(P,P^*):P\in \mathcal{Q}_{\{x,y\}}\}\subseteq \mathcal{B}_{\{x,y\}}$ and $\mathcal{B}'_{\{x',y'\}}=\{(P,P^*):P\in \mathcal{Q}_{\{x',y'\}}\}\subseteq \mathcal{B}_{\{x',y'\}}$  satisfy the outcomes of Lemma~\ref{lem:coolramsey}. In particular, for all distinct $\{x,y\},\{x',y'\}\in \binom{B}{2}$, it follows from Lemma~\ref{lem:coolramsey}\ref{lem:coolramsey_a} that for every $P\in \mathcal{Q}_{\{x,y\}}$ and every $P'\in \mathcal{Q}_{\{x',y'\}}$, we have $P^*\cap P'=\emptyset$. Equivalently, we have $V(\mathcal{Q}^*_{\{x,y\}})\cap V(\mathcal{Q}_{\{x',y'\}})=\emptyset$. Hence, $B$ is a $d$-short strong $k$-block in $G$ with respect to $\{\mathcal{Q}_{\{x,y\}}:\{x,y\}\in \binom{B}{2}\}\}$. This completes the proof of Theorem~\ref{thm:short}.
\end{proof}

Recall that a \textit{subdivision} of a graph $H$ is a graph $H'$ obtained from $H$ by replacing the edges of $H$ with
pairwise internally disjoint paths of non-zero lengths between the corresponding ends. Let $r\in \poi\cup \{0\}$. An \textit{$(\leq r)$-subdivision of} $H$ is a subdivision of $H$
in which the path replacing each edge has length at most $r+1$. Also, a \textit{proper subdivision} of $H$ is a subdivision of $H$
in which the path replacing each edge has length at least two. We need the following immediate corollary of a result of \cite{dvorak}:

\begin{theorem}[Dvo\v{r}\'{a}k \cite{dvorak}]\label{dvorak}
For every graph $H$ and all $d\in \poi\cup \{0\}$ and $t\in \poi$,  there is a constant $m=m(H,d,t)\in \poi$ with the following property. Let $G$ be a graph with no induced subgraph isomorphic to a subdivision of $H$. Assume that $G$ contains a $(\leq d)$-subdivision of $K_m$ as a subgraph. Then $G$ contains either $K_{t+1}$ or $K_{t,t}$.
\end{theorem}

From Theorems~\ref{thm:short} and \ref{dvorak} together, we deduce that:

\begin{theorem}\label{thm:noshort}
For all $d,t\in \poi$,  there is a constant $\eta=\eta(d,t)\in \poi$ with the following property. Let $G$ be a $t$-clean graph. Then there is no $d$-short $\eta$-block in $G$.
\end{theorem}
\begin{proof}
Let $m=m(W_{t\times t},d,t)\in \poi$ be as in Theorem~\ref{dvorak}. We show that  $\eta=\eta(d,t)=\max\{\upsilon(d,m), m\}$ satisfies the theorem, where $\upsilon(\cdot,\cdot)$ is as in Theorem~\ref{thm:short}. Let $G$ be a $t$-clean graph; that is, $G$ has no induced subgraph isomorphic to a $t$-basic obstruction, and in particular a subdivision of $W_{t\times t}$. Suppose for a contradiction that there is a $d$-short $\eta$-block $B$ in $G$. Since $|B|\geq m$, we may choose $B'\subseteq B$ with $|B'|=m$. It follows from Theorem~\ref{thm:short} that $B'$ is a $d$-short strong $m$-block in $G$. For every $2$-subset $\{x,y\}$ of $B'$, let $\mathcal{P}_{\{x,y\}}$ be a collection of $m\in \poi$ pairwise internally disjoint paths in $G$ from $x$ to $y$ as in the definition of a strong $m$-block, and fix a path $P_{\{x,y\}}\in \mathcal{P}_{\{x,y\}}$. Now the union of the paths $P_{\{x,y\}}$ for all $\{x,y\}\in \binom{B'}{2}$ forms a subgraph of $G$ isomorphic to a $(\leq d)$-subdivision of $K_m$. This, together with the choice of $m$ and the assumption that $G$ is $t$-clean, violates Theorem~\ref{dvorak}. This contradiction completes the proof of Theorem~\ref{thm:noshort}.
\end{proof}

\section{Obtaining complete bipartite minor models}\label{sec:model}
The main result of this section, Theorem~\ref{thm:grid}, shows that $t$-clean graphs of sufficiently large treewidth contain large complete bipartite minor models in which every ``branch set'' is a path, such that for every branch set $P$ on at least two vertices, each vertex in $P$ has neighbors in only a small number of branch sets in the ``opposite side'' of the bipartition. The exact statement of \ref{thm:grid} is somewhat technical and involves a number of definitions which we give below.

Let $G$ be a graph. For $w\in \poi$, a \textit{$w$-polypath in $G$} is a set $\mathcal{W}$ of $w$ pairwise disjoint paths in $G$. Let $\mathcal{W}$ be a $w$-polypath in $G$. For $d\in \poi$, we say $\mathcal{W}$ is \textit{$d$-loose} if for every $W\in \mathcal{W}$, each vertex $v\in W$ has neighbors (in $G$) in fewer than $d$ paths in $\mathcal{W}\setminus \{W\}$ (this is particular implies that a vertex of ``large degree'' has ``most'' of its neighbors on just one path; we will use this property extensively in Section~\ref{sec:turbinconst}). Also, for $w'\in \poi_{w}$, we say $\mathcal{W}$ is \textit{$w'$-fancy} if there exists $\mathcal{W}'\subseteq \mathcal{W}$ with $|\mathcal{W}'|=w'$ such that for every $W'\in \mathcal{W}'$ and every $W\in \mathcal{W}\setminus \mathcal{W}'$, $W'$ is not anticomplete to $W$ in $G$. It follows that if $\mathcal{W}$ is a $w$-polypath in $G$ which is $w'$-fancy, then $G[V(\mathcal{W})]$ has a $K_{w',w-w'}$-minor (see Figure~\ref{fig:polypathandconstellation}). 

For $s,l\in \poi$, an \textit{$(s,l)$-cluster in $G$} is a pair $(S,\mathcal{L})$ where $S\subseteq V(G)$ with $|S|=s$ and $\mathcal{L}$ is an $l$-polypath in $G\setminus S$, such that every vertex $x\in S$ has at least one neighbor in every path $L\in \mathcal{L}$. If $l=1$, say $\mathcal{L}=\{L\}$, we also denote the $(s,1)$-cluster $(S,\mathcal{L})$ by the pair $(S,L)$. For $d\in \poi$, we say $(S,\mathcal{L})$ is \textit{$d$-meager} if every vertex in $V(\mathcal{L})$ has fewer than $d$ neighbors in $S$. Again, it is easily seen that if $(S,\mathcal{L})$ is an $(s,l)$-cluster in $G$, then $G[S\cup V(\mathcal{L})]$ has a $K_{s,l}$-minor (see Figure~\ref{fig:polypathandconstellation}).

    \begin{figure}[t!]
        \centering
\includegraphics[scale=0.6]{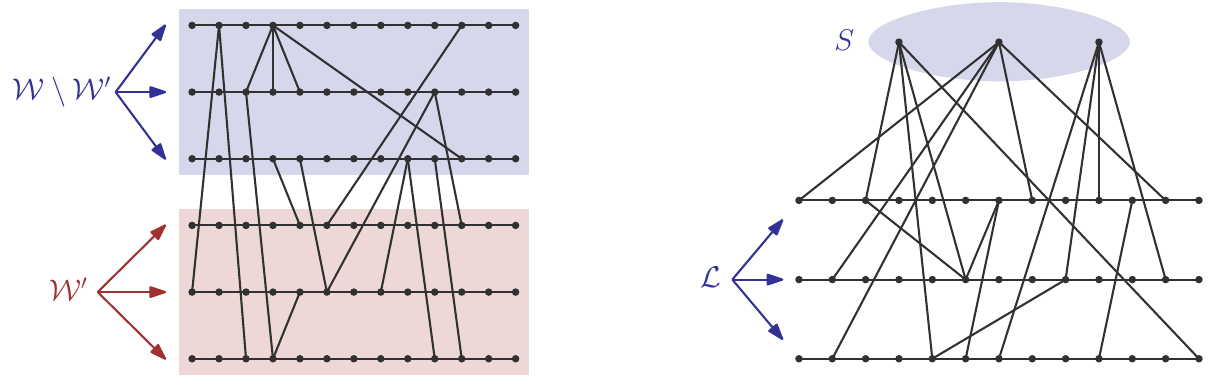}
        \caption{A $6$-polypath $\mca{W}$ which is $3$-fancy and $2$-loose (left) and a $(3,3)$-constellation $(S,\mca{L})$ which is $2$-ample (right).}
    \label{fig:polypathandconstellation}
    \end{figure}

Our goal is to prove:

\begin{theorem}\label{thm:grid}
    For all $c,l,s,t,w\in \poi$,  there is a constant $\gamma=\gamma(c,l,s,t,w)\in \poi$ with the following property. Let $G$ be a $t$-clean graph of treewidth more than $\gamma$. Then one of the following holds.
    \begin{enumerate}[\rm (a)]
        \item\label{thm:grid_a} There exists a $t^t$-meager $(s,l)$-cluster in $G$.
        \item\label{thm:grid_b} There exists a $2w$-polypath in $G$ which is both $w$-fancy and $(l+t^ts^{t^t})$-loose.
    \end{enumerate}
\end{theorem}

The proof of Theorem~\ref{thm:grid} is split into a number of lemmas. To begin with, we need the following multicolor version of Ramsey's Theorem for (uniform) hypergraphs.

\begin{theorem}[Ramsey \cite{multiramsey}]\label{multiramsey}
For all $n,q,r\in \poi$,  there is a constant $\rho(n,q,r)\in \poi$ with the following property. Let  $U$ be a set of cardinality at least $\rho(n,q,r)$ and let $W$ be a non-empty set of cardinality at most $r$. Let $\Phi:\binom{U}{q}\rightarrow W$ be a map. Then there exist $i\in W$ and $Z\subseteq U$ with $|Z|=n$ such that for every $A\in \binom{Z}{q}$, we have $\Phi(A)=i$.
\end{theorem}

For graphs (and two colors), it is also convenient to use a quantified version of Ramsey's classical result (recall that a \textit{stable set} is a set of pairwise non-adjacent vertices).

\begin{theorem}[Ramsey \cite{multiramsey}]\label{classicalramsey}
For all $c,s\in \poi$, every graph $G$ on at least $s^c$ vertices contains either a clique of cardinality $c+1$ or a stable set of cardinality $s$. 
\end{theorem}

From Theorem~\ref{multiramsey}, we deduce:

\begin{lemma}\label{lem:loosepolypath}
        For all $l,q,s\in \poi$,  there is a constant $\sigma=\sigma(l,q,s)\in \poi$ with the following property. Let $G$ be a graph and let $\mathcal{P}$ be a $\sigma$-polypath in $G$. Then one of the following holds. 
\begin{enumerate}[\rm (a)]
        \item\label{lem:loosepolypath_a} There exists an $(s,l)$-cluster $(S,\mathcal{L})$ in $G$ such that $S\subseteq V(\mathcal{P})$ and $\mathcal{L}\subseteq \mathcal{P}$.
        \item\label{lem:loosepolypath_b} There exists an $l$-loose $q$-polypath $\mathcal{Q}$ in $G$ with $\mathcal{Q}\subseteq \mathcal{P}$.
    \end{enumerate}
\end{lemma}

\begin{proof}
We claim that $$\sigma=\sigma(l,q,s)=\rho(\max\{l+s,q\},l+1,2^{l+1})$$
satisfies the lemma, where $\rho(\cdot,\cdot,\cdot)$ comes from Theorem~\ref{multiramsey}.

Assume that \ref{lem:loosepolypath}\ref{lem:loosepolypath_a} does not hold. Fix an enumeration $\mathcal{P}=\{P_1,\ldots, P_{\sigma}\}$ of the elements of $\mathcal{P}$. For every $(l+1)$-subset $T=\{t_1,\ldots, t_{l+1}\}$ of $\poi_{\sigma}$ with $t_1<\cdots<t_{l+1}$, let $\Phi(T)\subseteq \poi_{l+1}$ be the set of all $i\in \poi_{l+1}$ for which there exists a vertex $x_{t_i}\in P_{t_i}$ that has a neighbor in $P_{t_j}$ for every $j\in \poi_{l+1}\setminus \{i\}$. It follows that the map $\Phi:\binom{\poi_{\sigma}}{l+1}\rightarrow 2^{\poi_{l+1}}$ is well-defined. Therefore, by the choice of $\sigma$, we can apply Theorem~\ref{multiramsey} and obtain $Z\subseteq \poi_{\sigma}$ with $|Z|=\max\{l+s,q\}$ as well as $F\subseteq \poi_{l+1}$ such that for every $I\in \binom{Z}{l+1}$, we have $\Phi(I)=F$. Now we claim that:

\sta{\label{st:Fisempty} $F$ is empty.}

Suppose not. Then we may choose $f\in F\subseteq \poi_{l+1}$. Since $|Z|\geq l+s$, it follows that there exist $I,J,K\subseteq Z$ with $|I|=f-1$, $|J|=s$ and $|K|=l-f+1$, such that $\max I< \min J$ and  $\max J<\min K$. For every $j\in J$,  define $T_j=I\cup \{j\}\cup K$; it follows that $T_j\in \binom{Z}{l+1}$ and $j$ is the $f^{\text{th}}$ smallest element of $T_j$. In particular, for every $j\in J$, we have $\Phi(T_j)=F$ and so $f\in \Phi(T_j)$, which in turn implies that there is a vertex $x_{j}\in P_{j}$ that has a neighbor in $P_{t}$ for every $t\in T_j\setminus \{j\}=I\cup K$. Let $S=\{x_j:j\in J\}$ and let $\mathcal{L}=\{P_t:t\in I\cup K\}$. Then $|S|=s$, $\mathcal{L}$ is an $l$-polypath in $G\setminus S$, and every vertex in $S$ has a neighbor in every path in $\mathcal{L}$. But now $(S,\mathcal{L})$ is an $(s,l)$-cluster in $G$ satisfying \ref{lem:loosepolypath}\ref{lem:loosepolypath_a}, a contradiction. This proves \eqref{st:Fisempty}.

\medskip

Since $|Z|\geq q$, we may choose $Q\subseteq Z$ with $|Q|=q$. Let $\mathcal{Q}=\{P_i:i\in Q\}$. Then $\mathcal{Q}$ is a $q$-polypath in $G$ with $\mathcal{Q}\subseteq \mathcal{P}$, and by \eqref{st:Fisempty}, $\mathcal{Q}$ is $l$-loose. Hence, $\mathcal{Q}$ satisfies \ref{lem:loosepolypath}\ref{lem:loosepolypath_b}. This completes the proof of Lemma~\ref{lem:loosepolypath}.
\end{proof}

The following two lemmas have similar proofs:

\begin{lemma}\label{lem:meagercluster}
    Let $l,s,t\in \poi$, let $G$ be a graph and let $(S,\mathcal{L}_0)$ be an $(s,l+t^ts^{t^t})$-cluster in $G$. Then one of the following holds.
\begin{enumerate}[\rm (a)]
        \item\label{lem:meagercluster_a} $G$ contains $K_{t+1}$ or $K_{t,t}$. 
        \item\label{lem:meagercluster_b} There exists $\mathcal{L}\subseteq 
    \mathcal{L}_0$ with $|\mathcal{L}|=l$ such that $(S,\mathcal{L})$ is a $t^t$-meager $(s,l)$-cluster in $G$.
    \end{enumerate}
\end{lemma}
\begin{proof}
   Suppose that \ref{lem:meagercluster}\ref{lem:meagercluster_a} does not hold. For every $L\in \mathcal{L}_0$, let $t_{L}$ be the largest integer in $\poi_{s}$ for which there exists a vertex $u_L\in L$ with at least $t_{L}$ neighbors in $S$. It follows that:

    \sta{\label{st:fewtL} We have $|\{L\in \mathcal{L}_0:t_L\geq t^t\}|<t^ts^{t^t}$.}

    Suppose not. Let $\mathcal{S}\subseteq \{L\in \mathcal{L}_0:t_L\geq t^t\}$ with $|\mathcal{S}|=t^ts^{t^t}$. Then for every $L\in \mathcal{S}$, we may choose $u_L\in L$ and $T_L\subseteq S$ such that $|T_L|=t^t$ and $u_L$ is complete to $T_L$ in $G$. Since $|S|=s$, it follows that there exist $T\subseteq S$ and $\mathcal{T}\subseteq \mathcal{S}$ such that $|T|=|\mathcal{T}|=t^t$, and for every $L\in \mathcal{T}$, we have $T_L=T$. Let $U=\{u_L:L\in \mathcal{T}\}$. Then $T$ is disjoint from and complete to $U$ in $G$. Also, since $G$ is $K_{t+1}$-free and $|T|=|U|=t^t$, it follows from Theorem~\ref{classicalramsey} that there are two stable sets $T'\subseteq T$ and $U'\subseteq U$ in $G$ with $|T'|=|U'|=t$. But then $G[T'\cup U']$ is isomorphic to $K_{t,t}$, a contradiction. This proves \eqref{st:fewtL}.

    \medskip

    Now the result is immediate from \eqref{st:fewtL} and the fact that $|\mathcal{L}_0|=l+t^ts^{t^t}$. This completes the proof of Lemma~\ref{lem:meagercluster}. 
\end{proof}

\begin{lemma}\label{lem:loosepolypathbip}
    Let $l,s,w\in \poi$, let $G$ be a graph and let $\mathcal{Q},\mathcal{Q}'$ be $((s+1)^{l+1}w^{l^2})$-polypaths in $G$ with $V(\mathcal{Q})\cap V(\mathcal{Q}')=\emptyset$. Then one of the following holds. 
\begin{enumerate}[\rm (a)]
        \item\label{lem:loosepolypathbip_a} There exists an $(s,l)$-cluster $(S,\mathcal{L})$ in $G$ such that either $S\subseteq V(\mathcal{Q})$ and $\mathcal{L}\subseteq \mathcal{Q}'$, or $S\subseteq V(\mathcal{Q}')$ and $\mathcal{L}\subseteq \mathcal{Q}$.
        \item\label{lem:loosepolypathbip_b} There exist $\mathcal{W}\subseteq \mathcal{Q}$ and $\mathcal{W}'\subseteq \mathcal{Q}'$ with $|\mathcal{W}|=|\mathcal{W}'|=w$ such that every vertex in $V(\mathcal{W})$ has neighbors in fewer than $l$ paths in $\mathcal{W}'$, and every vertex in $V(\mathcal{W}')$ has neighbors in fewer than $l$ paths in $\mathcal{W}$.
    \end{enumerate}
\end{lemma}

\begin{proof}
Suppose \ref{lem:loosepolypathbip}\ref{lem:loosepolypathbip_a} does not hold. Let $r=sw^l+w$. Then we have 
$$(s+1)^{l+1}w^{l^2}=(s+1)((s+1)w^{l})^l=s(sw^{l}+w^l)^l+(s+1)^{l}w^{l^2}\geq sr^l+w\geq r.$$
In particular, one may choose $\mathcal{R}'\subseteq \mathcal{Q}'$ with $|\mathcal{R}'|=r$. For every $Q\in \mathcal{Q}$, let $r_{Q}$ be the largest integer in $\poi_{r}$ for which there exists a vertex $p_Q\in Q$ which has neighbors in at least $r_{Q}$ paths in $\mathcal{R}'$. It follows that:

    \sta{\label{st:polybipround1} We have $|\{Q\in \mathcal{Q}:r_Q\geq l\}|<sr^{l}$.}

    Suppose not. Let $\mathcal{P}\subseteq \{Q\in \mathcal{Q}:r_Q\geq l\}$ with $|\mathcal{P}|=sr^{l}$. Then for every $Q\in \mathcal{P}$, we may choose $p_Q\in Q$ and $\mathcal{L}_Q\subseteq \mathcal{R}'$ such that $|\mathcal{L}_Q|=l$ and $p_Q$ has a neighbor in every path in $\mathcal{L}_Q$. Since $|\mathcal{R}'|=r$, it follows that there exist $\mathcal{L}\subseteq \mathcal{R}'$ and $\mathcal{S}\subseteq \mathcal{P}$ such that $|\mathcal{L}|=l$, $|\mathcal{S}|=s$, and for every $Q\in \mathcal{S}$, we have $\mathcal{L}_Q=\mathcal{L}$. Let $S=\{p_Q:Q\in \mathcal{S}\}$; so we have $|S|=s$. But now $(S,\mathcal{L})$ is an $(s,l)$-cluster in $G$ with $S\subseteq V(\mathcal{S})\subseteq V(\mathcal{P})\subseteq V(\mathcal{Q})$ and $\mathcal{L} \subseteq \mathcal{R}'\subseteq \mathcal{Q}'$, contrary to the assumption that \ref{lem:loosepolypathbip}\ref{lem:loosepolypathbip_a} does not hold. This proves \eqref{st:polybipround1}.

    \medskip

    By \eqref{st:polybipround1} and since $|\mathcal{Q}|\geq sr^l+w$, we may choose $\mathcal{W}\subseteq \mathcal{Q}$ with $|\mathcal{W}|=w$ such that every vertex in $V(\mathcal{W})$ has neighbors in fewer than $l$ paths in $\mathcal{R}'$. 

Next, for every $R\in \mathcal{R}'$, let $w_{R}$ be the largest integer in $\poi_{w}$ for which there exists a vertex $q_R\in R$ which has neighbors in at least $w_{R}$ paths in $\mathcal{W}$.

    \sta{\label{st:polybipround2} We have $|\{R\in \mathcal{R}':w_R\geq l\}|<sw^l$.}

    Suppose not. Let $\mathcal{P}\subseteq \{R\in \mathcal{R}':w_R\geq l\}$ with $|\mathcal{P}|=sw^{l}$. Then for every $R\in \mathcal{P}$, we may choose $p_R\in R$ and $\mathcal{L}_R\subseteq \mathcal{W}$ such that $|\mathcal{L}_R|=l$ and $p_R$ has a neighbor in every path in $\mathcal{L}_R$. Since $|\mathcal{W}|=w$, it follows that there exist $\mathcal{L}\subseteq \mathcal{W}$ and $\mathcal{S}\subseteq \mathcal{P}$ such that $|\mathcal{L}|=l$, $|\mathcal{S}|=s$, and for every $R\in \mathcal{S}$, we have $\mathcal{L}_R=\mathcal{L}$. Let $S=\{q_R:R\in \mathcal{S}\}$; so we have $|S|=s$. But then $(S,\mathcal{L})$ is an $(s,l)$-cluster in $G$ with $S\subseteq V(\mathcal{S})\subseteq V(\mathcal{P})\subseteq V(\mathcal{R}')\subseteq V(\mathcal{Q}')$ and $\mathcal{L} \subseteq \mathcal{W}\subseteq \mathcal{Q}$, contrary to the assumption that \ref{lem:loosepolypathbip}\ref{lem:loosepolypathbip_a} does not hold. This proves \eqref{st:polybipround2}.

    \medskip

In view of \eqref{st:polybipround2} and since $|\mathcal{R}'|=r=sw^l+w$, we may choose $\mathcal{W}'\subseteq \mathcal{R}'\subseteq \mathcal{Q}'$ with $|\mathcal{W}'|=w$ such that every vertex in $V(\mathcal{W}')$ has neighbors in fewer than $l$ paths in $\mathcal{W}$. Now $\mathcal{W},\mathcal{W}'$ satisfy \ref{lem:loosepolypathbip}\ref{lem:loosepolypathbip_b}. This completes the proof of Lemma~\ref{lem:loosepolypathbip}.
\end{proof}

The last tool we need is a recent result from \cite{pinned}, the statement of which requires another definition. Let $G$ be a graph and let $w\in \poi$. By a \textit{$w$-web in $G$} we mean a pair $(W,\Lambda)$ where
\begin{enumerate}[(W1), leftmargin=15mm, rightmargin=7mm]
     \item\label{W1} $W$ is a $w$-subset of $V(G)$;
     \item\label{W2} $\Lambda:\binom{W}{2}\rightarrow 2^{V(G)}$ is a map such that for every $2$-subset $\{x,y\}$ of $W$, $\Lambda(\{x,y\})=\Lambda_{\{x,y\}}$ is a path in $G$ with ends $x,y$; and
     \item\label{W3} for all distinct $2$-subsets $\{x,y\},\{x',y'\}$ of $W$, we have $\Lambda_{\{x,y\}}\cap \Lambda_{\{x',y'\}}=\{x,y\}\cap \{x',y'\}$.
 \end{enumerate}
 It follows that $G$ has a subgraph isomorphic to a subdivision of $K_w$ if and only if there is a $w$-web in $G$. It is proved in \cite{pinned} that:

\begin{theorem}[Hajebi \cite{pinned}]\label{thm:pinned}
    For all $a,b,c,s\in \poi$,  there is a constant $\theta=\theta(a,b,c,s)\in \poi$ with the following property. Let $G$ be a graph and let $(W,\Lambda)$ be a $\theta$-web in $G$. Then one of the following holds.

     \begin{enumerate}[\rm (a)]
         \item\label{thm:pinned_a} There exists $A\subseteq W$ with $|A|=a$ and a collection $\mathcal{B}$ of pairwise disjoint $2$-subsets of $W\setminus A$ with $|\mathcal{B}|=b$ such that for every $x\in A$ and every $\{y,z\}\in \mathcal{B}$, $x$ has a neighbor in $\Lambda_{y,z}$.
          \item\label{thm:pinned_b} There are disjoint subsets $\mathcal{C}$ and $\mathcal{C}'$ of $\binom{W}{2}$ with $|\mathcal{C}|=|\mathcal{C}'|=c$ such that for every $\{x,y\}\in \mathcal{C}$ and every $\{x',y'\}\in \mathcal{C}'$, $\Lambda_{\{x,y\}}^*$ is not anticomplete to $\Lambda_{\{x',y'\}}^*$ in $G$. 
        \item\label{thm:pinned_c} There exists $S\subseteq W$ with $|S|=s$ such that:
        \begin{itemize}
        \item[-] $S$ is a stable set in $G$;
         \item[-] for any three vertices $x,y,z\in S$, $x$ is anticomplete to $\Lambda^*_{y,z}$; and
            \item[-] for all distinct $2$-subsets $\{x,y\},\{x',y'\}$ of $S$, $\Lambda^*_{\{x,y\}}$ is anticomplete to $\Lambda^*_{\{x',y'\}}$.
        \end{itemize}
      \end{enumerate}
\end{theorem}

We are now in a position to prove our main result in this section, which we restate:

\setcounter{theorem}{0}

\begin{theorem}\label{thm:grid}
    For all $c,l,s,t,w\in \poi$,  there is a constant $\gamma=\gamma(c,l,s,t,w)\in \poi$ with the following property. Let $G$ be a $t$-clean graph of treewidth more than $\gamma$. Then one of the following holds.
    \begin{enumerate}[\rm (a)]
        \item\label{thm:grid_a} There exists a $t^t$-meager $(s,l)$-cluster in $G$.
        \item\label{thm:grid_b} There exists a $2w$-polypath in $G$ which is both $w$-fancy and $(l+t^ts^{t^t})$-loose.
    \end{enumerate}
\end{theorem}
\begin{proof}
 Let $b=l+t^ts^{t^t}$, let $\sigma=\sigma(b,(s+1)^{b+1}w^{b^2},s)$ be as in Lemma~\ref{lem:loosepolypath} and let $\theta=\theta(s,b,\sigma,2t^2)$ be as in Theorem~\ref{thm:pinned}. We claim that 
$$\gamma=\gamma(c,l,s,t,w)=\kappa(\theta,t)$$
satisfies the theorem, where $\kappa(\cdot,\cdot)$ is as in Theorem~\ref{thm:tw7block}. 
 
Let $G$ be a $t$-clean graph of treewidth more than $\gamma$. Suppose \ref{thm:grid}\ref{thm:grid_a} does not hold. This, along with the choice of $b$, the assumption that $G$ is $t$-clean, and Lemma~\ref{lem:meagercluster}, implies that:

\sta{\label{st:nobigcluster}There is no $(s,b)$-cluster in $G$.}

From the choice of $\gamma$ and Theorem~\ref{thm:tw7block}, we obtain a strong $\theta$-block $B$ in $G$. In particular,
for every $2$-subset $\{x,y\}$ of $B$, we may choose a path $\Lambda_{\{x,y\}}$ in $G$ from $x$ to $y$, such that for all distinct $2$-subsets $\{x,y\}, \{x',y'\}$ of $B$, we have $\Lambda_{\{x,y\}}\cap \Lambda_{\{x',y'\}}=\{x,y\}\cap \{x',y'\}$. It follows that $(B,\Lambda)$ is a $\theta$-web in $G$, and so we can apply Theorem~\ref{thm:pinned} to $(B,\Lambda)$. Note that Theorem~\ref{thm:pinned}\ref{thm:pinned_a} yields an $(s,b)$-cluster in $G$, which violates \eqref{st:nobigcluster}. Also, Theorem~\ref{thm:pinned}\ref{thm:pinned_c} implies that $G$ contains a proper subdivision of $K_{2t^2}$ (as an induced subgraph), which in turn contains a proper subdivision of every graph on $2t^2$ vertices. But this violates the assumption that $G$ is $t$-clean because $|V(W_{t\times t})|\leq 2t^2$.

We conclude that Theorem~\ref{thm:pinned}\ref{thm:pinned_b} holds. In particular, there are two $\sigma$-polypaths $\mathcal{C},\mathcal{C}'$ in $G$ such that for every $L\in \mathcal{C}$ and every $L'\in \mathcal{C}'$, $L$ is not anticomplete to $L'$ in $G$. Now, from \eqref{st:nobigcluster}, the choice of $\sigma$ and Lemma~\ref{lem:loosepolypath}, it follows that there are two $b$-loose $((s+1)^{b+1}w^{b^2})$-polypaths $\mathcal{Q}\subseteq \mathcal{C}$ and $\mathcal{Q}'\subseteq \mathcal{C}'$ in $G$. Furthermore, from \eqref{st:nobigcluster} and Lemma~\ref{lem:loosepolypathbip}, it follows that there exist $\mathcal{W}\subseteq Q$ and $\mathcal{W}'\subseteq Q'$ with $|\mathcal{W}|=|\mathcal{W}'|=w$ such that every vertex in $V(\mathcal{W})$ has neighbors in fewer than $b$ paths in $\mathcal{W}'$, and every vertex in $V(\mathcal{W}')$ has neighbors in fewer than $b$ paths in $\mathcal{W}$. But now $\mathcal{W}\cup \mathcal{W}'$ is a $2w$-polypath in $G$ which is both $w$-fancy and $b$-loose, and so \ref{thm:grid}\ref{thm:grid_b} holds. This completes the proof of Theorem~\ref{thm:grid}.
\end{proof}

\section{Dealing with complete bipartite minor models}\label{sec:turbinconst}

Here we take the penultimate step of our proof by showing that:

\begin{theorem}\label{mainloaded}
For all $c,t\in \poi$,  there is a constant $\Gamma=\Gamma(c,t)$ such that every $t$-clean graph $G$ of treewidth more than $\Gamma$ contains a $c$-hassle.
\end{theorem}

From Theorem~\ref{thm:grid}, we know that every $t$-clean graph of sufficiently large treewidth contains, omitting the corresponding parameters, either a meager cluster or a polypath which is both loose and fancy. So it suffices to prove Theorem~\ref{mainloaded} separately in each of these two cases. First, we show:

\begin{theorem}\label{thm:tasselincluster}
 Let $c,d\in \poi$ and let $G$ be a graph. Assume that there is a $d$-meager $(2cd,2c^2d)$-cluster in $G$. Then $G$ contains a $c$-hassle.
\end{theorem}

\begin{proof}
    Let $(S,\mathcal{L})$ be a $d$-meager $(2cd,2c^2d)$-cluster in $G$.  For every path $P$ in $G[V(\mathcal{L})]$, let $S_P$ be the set of all vertices in $S$ with a neighbor in $P$. For every $L\in \mathcal{L}$ and each end $u$ of $L$, let $L_u$ be the longest path in $L$ containing $u$ such that $|S_{L_u}|<cd$; then $u$ is an end of $L_u$. It follows that:

         \sta{\label{st:Tsmall} For every $L\in \mathcal{L}$ and every end $u$ of $L$, we have $|L_u|\geq c$.}

         Suppose for a contradiction that $|L_u|\leq c-1$. Then since $(S,\mathcal{L})$ is $d$-meager, it follows that $|S_{L_u}|<(c-1)d<|S|$. Let $v$ be the end of $L_u$ other than $u$. Since $|S_{L_u}|<|S|$ and every vertex in $S$ has a neighbor in $L$, it follows that $L\setminus L_u\neq \emptyset$. Let $v'$ be the unique neighbor of $v$ in $L\setminus L_u$, and let $P=u\dd L_u v\dd v'$. Then $|P|\leq c$, and since $(S,\mathcal{L})$ is $d$-meager, it follows that $|S_P|< cd$. This violates the choice of $L_u$, and proves \eqref{st:Tsmall}.

         \medskip
    
    From \eqref{st:Tsmall}, it follows that for every $L\in \mathcal{L}$, there are fewer than $2cd$ vertices in $S$ with a neighbor among the first or last $c$ vertices of $L$. This, along with $|S|=2cd$ and the fact that every vertex in $S$ has a neighbor in every path in $\mathcal{L}$, implies that for every $L\in \mathcal{L}$, there is a vertex $x_L\in S$ with a neighbor in $L$ and no neighbor among the first and last $c$ vertices of $L$. On the other hand, we have $|\mathcal{L}|=c|S|$. Thus, there exist $x\in S$ and $\mathcal{W}\subseteq \mathcal{L}$ with $|\mathcal{W}|=c$ such that for every $L\in \mathcal{W}$, the vertex $x$ has a neighbor in $L$ and no neighbor among the first and last $c$ vertices of $L$. But now $\Xi=G[\{x\}\cup V(\mathcal{W})]$ is a $c$-hassle in $G$ with neck $x$ and with $\mathcal{W}$ as its set of walks (each of which is a path in $\Xi$). This completes the proof of Theorem~\ref{thm:tasselincluster}.
\end{proof}

Handling the second outcome of Theorem~\ref{thm:grid} is more demanding.
We prove:

\begin{theorem}\label{thm:tasselinpoly}
 For all $c,d,t\in \poi$,  there is a constant $\Omega=\Omega(c,d,t)\in \poi$ with the following property. Let $G$ be a $t$-clean graph. Assume that there exists a $2\Omega$-polypath in $G$ which is both $\Omega$-fancy and $d$-loose. Then $G$ contains a $c$-hassle.
\end{theorem}

We need to prepare for the proof of Theorem~\ref{thm:tasselinpoly}. Let $G$ be a graph, let $x,y\in V(G)$ be distinct and non-adjacent, and let $\mathcal{P}$ be a collection of pairwise internally disjoint paths in $G$ from $x$ to $y$. An \textit{$x$-slash for $\mathcal{P}$ in $G$} is a path $W$ in $G\setminus \{x,y\}$ such that for every $P\in \mathcal{P}$, the unique neighbor of $x$ in $P$ belongs to $W$. Our first lemma says that:

\begin{lemma}\label{lem:slashnoshort}
Let $c,p,q\in \poi$. Let $G$ be a graph, let $x,y\in V(G)$ be distinct and non-adjacent, and let $\mathcal{P}_0$ be a collection of $c(p+1)q$ pairwise internally disjoint paths in $G$ from $x$ to $y$. Let $W_0$ be an $x$-slash for $\mathcal{P}_0$ in $G$. Then one of the following holds.
  \begin{enumerate}[\rm (a)]
    
        \item\label{lem:slashnoshort_a} There exists $\mathcal{P}\subseteq \mathcal{P}_0$ with $|\mathcal{P}|=p$ and a path $W\subseteq W_0$, such that:
        \begin{itemize}
            \item $W$ is an $x$-slash for $\mathcal{P}$ in $G$; and
            \item there is no path $Q$ of length at most $c+1$ in $V(\mathcal{P})\cup W$ from $x$ to $y$.
        \end{itemize}
    \item\label{lem:slashnoshort_b} There exists a collection $\mathcal{Q}$ of $q$ pairwise internally disjoint paths in $G$ from $x$ to $y$, each of length at most $c+1$.
    \end{enumerate}
  \end{lemma}
  \begin{proof}
Let $\mathcal{C}$ be a maximal set of pairwise internally disjoint paths in $G$ from $x$ to $y$, each of length at most $c+1$. In particular, for every $Q\in \mathcal{C}$, $Q^*$ is a path in $G$ on at most $c$ vertices.

Note that if $|\mathcal{C}|\geq q$, then \ref{lem:slashnoshort}\ref{lem:slashnoshort_b} holds, as required. Thus, we may assume that $|\mathcal{C}|<q$, and so $|V(\mathcal{C}^*)|<cq$.  In particular, $W_0\setminus \mathcal{C}^*$ has at most $cq$ components, and since the paths in $\mathcal{P}_0$ are pairwise internally disjoint, it follows that there are at most $|V(\mathcal{C}^*)|<cq$ paths $P$ in $\mathcal{P}_0$ for which $P\cap \mathcal{C}^*\neq \emptyset$. This, along with the assumption that $|\mathcal{P}_0|=c(p+1)q$ and $W_0$ is an $x$-slash for $\mathcal{P}_0$ in $G$, implies that there exist $\mathcal{P}\subseteq \mathcal{P}_0$ with $|\mathcal{P}|=p$ and $V(\mathcal{P})\cap \mathcal{C}^*=\emptyset$, as well as a component $W$ of $W_0\setminus \mathcal{C}^*\subseteq G\setminus \{x,y\}$, such that for every $P\in \mathcal{P}$, the unique neighbor of $x$ in $P$ belongs to $W$. In other words, $W$ is an $x$-slash for $\mathcal{P}$. Moreover, we have $(V(\mathcal{P})\cup W)\cap \mathcal{C}^*=\emptyset$, and so by the maximality of $\mathcal{C}$, there is no path $Q$ of length at most $c+1$ in $V(\mathcal{P})\cup W$ from $x$ to $y$. Now $\mathcal{P}$ and $W$ satisfy \ref{lem:slashnoshort}\ref{lem:slashnoshort_a}, as desired.
\end{proof}

  The next lemma is the main step  in the proof of Theorem~\ref{thm:tasselinpoly}.

\begin{lemma}\label{lem:slash}
  For all $c,q,t\in \poi$,  there is a constant $\mu=\mu(c,q,t)\in \poi$ with the following property. Let $G$ be a graph, let $x,y\in V(G)$ be distinct and non-adjacent, and let $\mathcal{P}_0$ be a collection of $\mu$ pairwise internally disjoint paths in $G$ from $x$ to $y$. Assume that there is an $x$-slash for $\mathcal{P}_0$ in $G$. Then one of the following holds.
  \begin{enumerate}[\rm (a)]
        \item\label{lem:slash_a} $G$ contains $K_{t+1}$ or $K_{t,t}$.
        \item\label{lem:slash_b} $G$ contains a $c$-hassle.
        \item\label{lem:slash_c} There is a collection $\mathcal{Q}$ of $q$ pairwise internally disjoint paths in $G$ from $x$ to $y$, each of length at most $c+1$.
    \end{enumerate}
  \end{lemma}
  \begin{proof}
   Let $b=b(c,t)=t^t(2c^2+(2ct^t)^{t^t})$. Let $\Upsilon(\cdot,\cdot,\cdot)$
be as in Lemma~\ref{lem:coolramsey}, and let
$$\Upsilon_1=\Upsilon(2c+2,1,c);$$
$$m=c\Upsilon_1;$$
$$\Upsilon_2=\Upsilon(c,b,3);$$
be as in Lemma~\ref{lem:coolramsey}. We will show that 
$$\mu=\mu(c,q,t)=c((c+2)m\Upsilon_2 +1)q$$
satisfies the lemma. Let $G$ be a graph, let $x,y\in V(G)$ be distinct and non-adjacent, let $\mathcal{P}_0$ be a collection of $\mu$ pairwise internally disjoint paths in $G$ from $x$ to $y$, and let $W_0$ be an $x$-slash for $\mathcal{P}_0$ in $G$.

Assume that none of the three outcomes of \ref{lem:slash} hold. We apply Lemma~\ref{lem:slashnoshort} to $x,y,\mathcal{P}_0$ and $W_0$. In particular, since \ref{lem:slash}\ref{lem:slash_c} is equivalent to  Lemma~\ref{lem:slashnoshort}\ref{lem:slashnoshort_b}, it follows from the choice of $\mu$ that:

\sta{\label{st:getP&W}There exists $\mathcal{P}\subseteq \mathcal{P}_0$ with $|\mathcal{P}|=(c+2)m\Upsilon_2$ and a path $W\subseteq W_0$ such that:
 \begin{itemize}
            \item $W$ is an $x$-slash for $\mathcal{P}$ in $G$; and
            \item there is no path $Q$ of length at most $c+1$ in $V(\mathcal{P})\cup W$ from $x$ to $y$.
        \end{itemize}}

From now on, let $\mathcal{P}$ and $W$ be as in \eqref{st:getP&W}. For every vertex $v\in N_{V(\mathcal{P})}(x)$, we denote by $P_v$ the unique path in $\mathcal{P}$ for which $v$ is the unique neighbor of $x$ in $P_v$. Since $W$ is an $x$-slash for $\mathcal{P}$, it follows that $N_{V(\mathcal{P})}(x)\subseteq W$. Let $w_1$ and $w_2$ be the ends of $W$. Since $|\mathcal{P}|=cm\Upsilon_2+2m\Upsilon_2$, it follows that one may choose $3m$ pairwise disjoint $\Upsilon_2$-subsets $\{U_{1,i},V_i,U_{2,i}:i\in \poi_{m}\}$ of $N_{V(\mathcal{P})}(x)\subseteq W$, such that the following hold.

\begin{itemize}
\item For each $i\in \poi_{m}$, there are $\Upsilon_2$ pairwise disjoint paths $\{A_v:v\in V_i\}$ in $W$, each on $c$ vertices, such that for every $v\in V_i$:
    \begin{itemize}
    \item $A_v$ contains $v$; and
    \item Traversing $W$ from $w_1$ to $w_2$, every vertex in $U_{1,i}$ appears before every vertex in $A_v$, and every vertex in $A_v$ appears before every vertex in $U_{2,i}$.
\end{itemize}
In particular, traversing $W$ from $w_1$ to $w_2$,  every vertex in $U_{1,i}$ appears before every vertex in $V_i$, and every vertex in $V_i$ appears before every vertex in $U_{2,i}$ (see Figure~\ref{fig:segmentpath}).

    \item  For every $i\in \poi_{m-1}$, traversing $W$ from $w_1$ to $w_2$, every vertex in $U_{2,i}$ appears before every vertex in $U_{1,i+1}$ (see Figure~\ref{fig:segmentpath2}).
\end{itemize}

\begin{figure}[t!]
\centering
\includegraphics[scale=0.6]{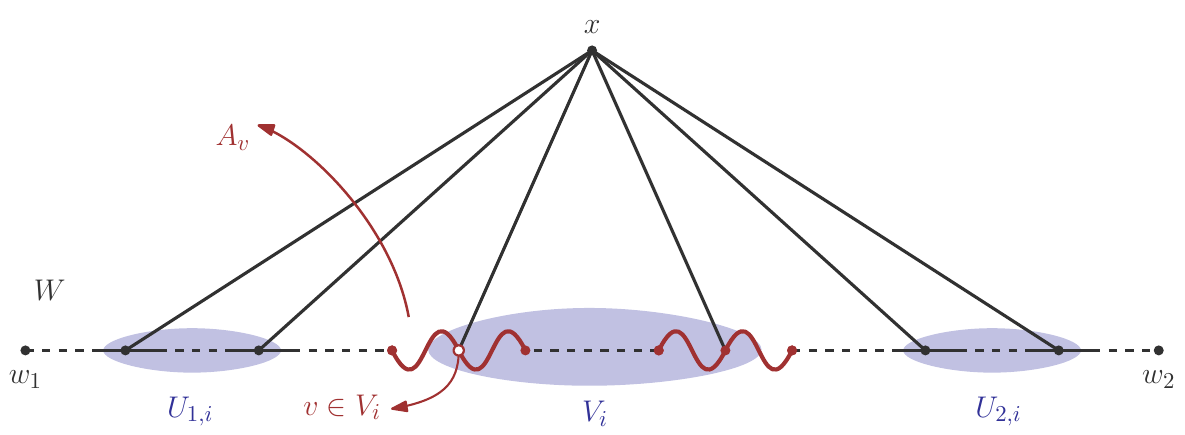}
\caption{The subsets $U_{1,i},V_i,U_{2,i}$ of $N_{V(\mathcal{P})}(x)\subseteq W$ and the paths $\{A_v:v\in V_i\}$.}
\label{fig:segmentpath}
\end{figure}

\begin{figure}[t!]
\centering
\includegraphics[scale=0.6]{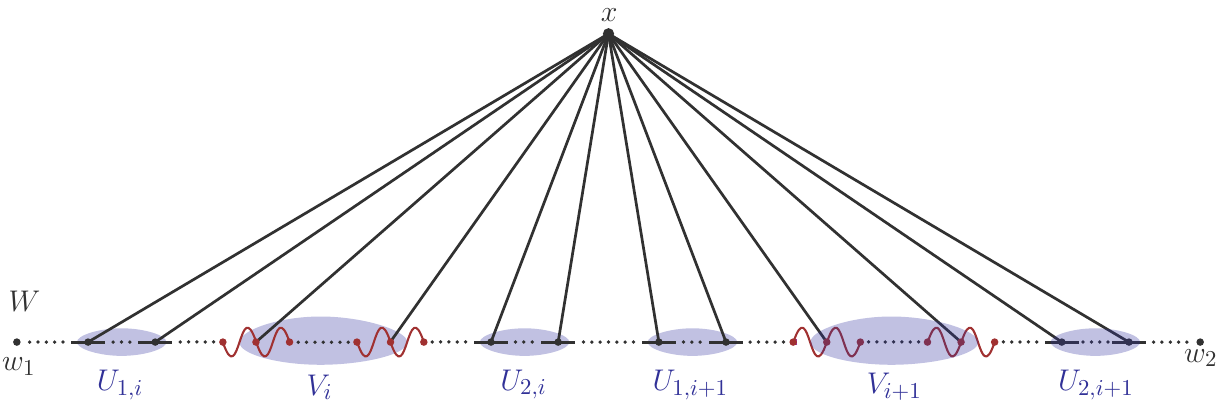}\caption{The subsets $U_{1,i},V_i,U_{2,i}, U_{1,i+1},V_{i+1},U_{2,i+1}$ of $N_{V(\mca{P})}(x)\subseteq W$.}
\label{fig:segmentpath2}
\end{figure}

We deduce that:

\sta{\label{st:disjointantislash} For every $i\in \poi_{m}$, there exist $u_{1,i}\in U_{1,i}$, $v_i\in V_i$ and $u_{2,i}\in U_{2,i}$, for which $A_{v_i}$ is disjoint from $P_{u_{1,i}}\cup P_{u_{2,i}}$ and anticomplete to $(P_{u_{1,i}}\cup P_{u_{2,i}})\setminus \{x\}$.}

To see this, let
$$\mathcal{A}_i=\{(A_v,\emptyset):v\in V_i\};$$  $$\mathcal{B}_{1,i}=\{(\emptyset,P_u^*):u\in U_{1,i}\};$$ 
$$\mathcal{B}_{2,i}=\{(\emptyset,P_u^*):u\in U_{2,i}\}.$$
Then $\mathcal{A}_i,\mathcal{B}_{1,i},\mathcal{B}_{2,i}$ are three collections of pairwise disjoint $c$-pairs over $V(G)$, each of cardinality $\Upsilon_2$. By the choice of $\Upsilon_2$, we can apply Lemma~\ref{lem:coolramsey} to $\mathcal{A}_i,\mathcal{B}_{1,i},\mathcal{B}_{2,i}$. It follows that there exist $U'_{1,i}\subseteq U_{1,i}$, $V'_i\subseteq V_i$ and $U'_{2,i}\subseteq U_{2,i}$ with $|U'_{1,i}|=|V'_i|=|U'_{2,i}|=b$ such that the collections
$\mathcal{A}'_i=\{(A_v,\emptyset):v\in V'_i\}$, $\mathcal{B}'_{1,i}=\{(\emptyset,P_u^*):u\in U'_{1,i}\}$ and 
$\mathcal{B}'_{2,i}=\{(\emptyset,P_u^*):u\in U'_{2,i}\}$ satisfy Lemma~\ref{lem:coolramsey}\ref{lem:coolramsey_a} and \ref{lem:coolramsey}\ref{lem:coolramsey_b}. In fact,  Lemma~\ref{lem:coolramsey}\ref{lem:coolramsey_a} along with the assumption that $x,y\notin W$, implies that for every $u_{1}\in U'_{1,i}$, every $v\in V'_i$ and every $u_{2}\in U'_{2,i}$, we have $A_{v}\cap (P_{u_{1}}\cup P_{u_{2}})=\emptyset$. It remains to show that there exist $u_{1,i}\in U'_{1,i}$, $v_i\in V'_i$ and $u_{2,i}\in U'_{2,i}$, for which $A_{v_i}$ is anticomplete to $(P_{u_{1,i}}\cup P_{u_{2,i}})\setminus \{x\}$. Suppose not. Note that for every $v\in V'_i$, since $v\in A_v$ is a neighbor of $x$, it follows from the second bullet of \eqref{st:getP&W} that $y$ is anticomplete to $A_v$. Consequently, by Lemma~\ref{lem:coolramsey}\ref{lem:coolramsey_b}, there exists $j\in \{1,2\}$ such that for every $v\in V_{i}$, there exists a vertex $x_v\in A_v$ which has a neighbor in $P^*_u$ for every $u\in U'_{j,i}$. On the other hand, since $b\geq 2ct^t$, it follows that there exists $V\subseteq V'_i$ with $|V|=2ct^t$. Now, let $S=\{x_v:v\in V\}$ and let $\mathcal{L}=\{P_u^*:u\in U'_{j,i}\}$. Then $(S,\mathcal{L})$ is a $(2ct^t,b)$-cluster in $G$. This, combined with the choice of $b$, Lemma~\ref{lem:meagercluster}, and the assumption that $G$ is $(K_{t+1},K_{t,t})$-free, implies that there exists a $t^t$-meager $(2ct^t,2c^2t^t)$-cluster in $G$. But then by Theorem~\ref{thm:tasselincluster}, $G$ contains a $c$-hassle. This violates the assumption that \ref{lem:slash}\ref{lem:slash_b} does not hold, and so proves \eqref{st:disjointantislash}.

\medskip

Henceforth, for each $i\in \poi_{m}$, let $u_{1,i}\in U_{1,i}$, $v_i\in V_i$ and $u_{2,i}\in U_{2,i}$ be as in \eqref{st:disjointantislash}. We write $A_i=A_{v_i}$, $P_{1,i}=P_{u_{1,i}}$ and $P_{2,i}=P_{u_{2,i}}$. Also, we denote by $a_{1,i},a_{2,i}$ the ends of $A_{i}$, such that $W$ traverses the vertices $w_1,a_{1,i},a_{2,i},w_2$ in this order. It follows that $W$ traverses the vertices $w_1,u_{1,i},a_{1,i},a_{2,i},u_{2,i},w_2$ in this order, and $a_{1,i},a_{2,i}$ are the only vertices among $u_{1,i},a_{1,i},a_{2,i},u_{2,i}$ which can be the same (only if $c=1$).

Let $i\in \poi_{m}$ be fixed. For each $j\in \{1,2\}$, traversing $a_{j,i}\dd W\dd u_{j,i}$ starting at $a_{j,i}$, let $w'_{j,i}$ be the first vertex in $a_{j,i}\dd W\dd u_{j,i}$ with a neighbor in $P_{j,i}\setminus \{x\}$ (note that $w'_{j,i}$ exists because $u_{j,i}$ has a neighbor in $P_{j,i}\setminus \{x\}$). From \eqref{st:disjointantislash},  we know that $A_i$ is disjoint from and anticomplete to $(P_{1,i}\cup P_{2,i})\setminus \{x\}$; in particular, for every $j\in \{1,2\}$, $u_{j, i} \in P_{j, i} \setminus \{x\}$ has a neighbor in the interior of $a_{j,i}\dd W\dd u_{j, i}$. It follows that for every $j\in \{1,2\}$, the vertex $w'_{j,i}$ belongs to the interior of $a_{j,i}\dd W\dd u_{j,i}$, and there is a path $R_{j,i}$ in $G$ from $w'_{j,i}$ to $y$ whose interior is contained in $P_{j,i}^*$.

For every $i\in \poi_{m}$ and each $j\in \{1,2\}$, let $R'_{j,i}$ be the longest path of length at most $c+1$ in $R_{j,i}\setminus \{y\}$ containing $w'_{j,i}$. It follows that:

\sta{\label{st:pathininterior} For every $i\in \poi_{m}$ and each $j\in \{1,2\}$, we have $R'_{j,i}\setminus \{w'_{j,i}\}\subseteq P^*_{j,i}$. Consequently, the sets $\{R'_{1,i}\cup R'_{2,i}:i\in \poi_{m}\}$ are pairwise disjoint in $G$.}

Observe that by the choice of $R'_{j,i}$, either $R'_{j,i}=R_{j,i}\setminus \{y\}$, in which case \eqref{st:pathininterior} follows immediately from the fact that $R_{j,i}^*\subseteq P_{j,i}^*$, or $R'_{j,i}$ is a path on $c+2$ vertices with $R'_{j,i}\setminus \{w'_{j,i}\}\subseteq P^*_{j,i}$. This proves \eqref{st:pathininterior}.

\medskip

Now, for each $i\in \poi_{m}$, let $W'_i=w'_{1,i}\dd W\dd w'_{2,i}$, and let $W_i=G[R'_{1,i}\cup W'_i\cup R'_{2,i}]$ be the walk such that traversing $W_i$ from $\varphi_{W_i}(1)$ to $\varphi_{W_i}(n_{W_i})$, we first traverse the path $R'_{1,i}$ starting at the end distinct from $w'_{1,i}$ and stopping at $w'_{1,i}$, then we traverse the path $W'_i$ from $w'_{1,i}$ to $w'_{2,i}$, and then we traverse $R'_{2,i}$ starting at $w'_{2,i}$ (so we have $n_{W_i}=|R'_{1,i}|+|W'_i|+|R'_{2,i}|-2$). In particular, we have $W_i\subseteq V(\mathcal{P}^*)\cup W$. We claim that:

\sta{\label{st:Wi's} The following hold.
\begin{itemize}
    \item For all $i\in \poi_{m}$, the walk $W_i$ is $c$-stretched and $x$ has a neighbor in $W_i$.
    \item For all $i\in \poi_{m}$, the vertex $x$ has no neighbors among the first and last $c$ vertices of $W_i$.
    \item The paths $\{W'_i: i\in \poi_{m}\}$ are pairwise disjoint.
    \item The sets $\{W_i\setminus W'_i: i\in \poi_{m}\}$ are pairwise disjoint.
\end{itemize}}
The first bullet is immediate from $|A_i|=c$ and the observation that both $R_{1,i}'\dd w'_{1,i}\dd W\dd a_{2,i}$ and $R'_{2,i}\dd w'_{1,i}\dd W\dd a_{1,i}$ are paths in $G$ containing $A_i$. Also, the third bullet is trivial, and the fourth is immediate from \eqref{st:pathininterior}. It remains to prove the second bullet. Note that by the definition of $R'_{1,i}$ and $R'_{2,i}$, either $y$ has a neighbor among the first or the last $c$ vertices of $W_i$, or the first $c+1$ vertices of $W_i$ are contained in $P^*_{1,i}$, and the last $c+1$ vertices of $W_i$ are contained in $P^*_{2,i}$. In the former case, the result follows directly from the second bullet of \eqref{st:getP&W}, and in the latter case, the result follows from  the fact that $x$ has exactly one neighbor in $P^*_{1,i}$ and exactly one neighbor in $P^*_{2,i}$. This proves \eqref{st:Wi's}.

\medskip

We can now finish the proof. For every $k\in \poi_{c}$, let 
$$I_k=\{i\in \poi_{m}: i= k \textrm{ (mod } c)\};$$
$$\mathcal{B}_k=\{(W_i\setminus W'_i, W'_i): i\in I_k\}.$$
From the choice of $m$ and the third bullet of \eqref{st:Wi's}, it follows that $\mathcal{B}_1,\ldots, \mathcal{B}_c$ are collections of pairwise disjoint $(2c+2)$-pairs over $V(G)$, each of cardinality $\Upsilon_1$. The choice of $\Upsilon_1$ in turn allows for an application of Lemma~\ref{lem:coolramsey} to $\mathcal{B}_1,\ldots, \mathcal{B}_c$. In particular,  Lemma~\ref{lem:coolramsey}~\ref{lem:coolramsey_a} implies that for every $k\in \poi_{c}$, there exists $i_k\in I_k$ such that  $W_{i_k}\setminus W'_{i_k}$ and $ W_{i_l}$ are disjoint for all distinct $k,l\in \poi_{c}$. This, combined with the third and the fourth bullet of \eqref{st:Wi's}, implies that $W_{i_1},\ldots,W_{i_c}$ are pairwise disjoint. But now by the first two bullets of \eqref{st:Wi's}, the subgraph of $G$ induced on $\{x\}\cup W_{i_1}\cup \cdots \cup W_{i_c}$ is a $c$-hassle with neck $x$ and walks $W_{i_1},\ldots,W_{i_c}$, contradicting the assumption that \ref{lem:slash}\ref{lem:slash_b} does not hold. This completes the proof of Lemma~\ref{lem:slash}.
\end{proof}

From Lemma~\ref{lem:slash}, we deduce the following:
  
  \begin{lemma}\label{lem:bananinpoly}
  For all $c,d,q,t\in \poi$,  there is a constant $\zeta=\zeta(c,d,q,t)\in \poi$ with the following property. Let $G$ be a graph and let $\mathcal{W}$ be a $d$-loose polypath in $G$. Let $x,y\in V(\mathcal{W})$ be distinct and non-adjacent, and let $\mathcal{P}$ be a collection of $\zeta$ pairwise internally disjoint paths in $G$ from $x$ to $y$ with $V(\mathcal{P})\subseteq V(\mathcal{W})$. Then one of the following holds.
  \begin{enumerate}[\rm (a)]
        \item\label{lem:bananinpoly_a} $G$ contains $K_{t+1}$ or $K_{t,t}$.         \item\label{lem:bananinpoly_b} $G$ contains a $c$-hassle.
        \item\label{lem:bananinpoly_c} There exists a collection $\mathcal{Q}$ of $q$ pairwise internally disjoint paths in $G$ from $x$ to $y$, each of length at most $c+1$.
    \end{enumerate}
  \end{lemma}
  \begin{proof}
Let $\mu=\mu(c,q,t)\in \poi$
be as in Lemma~\ref{lem:slash}. We show that 
$$\zeta(c,d,q,t)=2d\mu$$
satisfies the lemma. Let $W\in \mathcal{W}$ such that $x\in W$. Then $x$ has at most two neighbors in $W$. Also, since $\mathcal{W}$ is $d$-loose, it follows that $x$ has neighbors in at most $d-1$ paths in $\mathcal{W}\setminus \{W\}$. This, along with the fact that $|\mathcal{P}|=\zeta\geq 2(d-1)\mu+2$, implies that there exists $\mathcal{P}_1\subseteq \mathcal{P}$ with $|\mathcal{P}_1|=2\mu$ and a path $W_1\in \mathcal{W}\setminus \{W\}$, such that for every $P\in \mathcal{P}_1$, the unique neighbor of $x$ in $P$ belongs to $W_1$. In particular, we have $x\notin W_1$ because $W,W_1\in \mathcal{W}$ are distinct (hence disjoint) and $x\in W$. Since $W_1\setminus \{y\}$ has one or two components (depending on whether $y\in W^*_1$ or not), it follows that there exist $\mathcal{P}_0\subseteq \mathcal{P}_1$ with $|\mathcal{P}_0|=\mu$ as well as a component $W_0$ of $W_1\setminus \{y\}$, such that for every $P\in \mathcal{P}_0$, the unique neighbor of $x$ in $P$ belongs to $W_0$. Therefore, $\mathcal{P}_0$ is a collection of $\mu$ pairwise internally disjoint paths in $G$ from $x$ to $y$, and $W_0$ is an $x$-slash for $\mathcal{P}_0$ in $G$. Now the result follows from Lemma~\ref{lem:slash} applied to $\mathcal{P}_0$.
\end{proof}

  We can now restate and prove Theorem~\ref{thm:tasselinpoly}:
  
\setcounter{theorem}{2}
\begin{theorem}\label{thm:tasselinpoly}
 For all $c,d,t\in \poi$,  there is a constant $\Omega=\Omega(c,d,t)\in \poi$ with the following property. Let $G$ be a $t$-clean graph. Assume that there exists a $2\Omega$-polypath in $G$ which is both $\Omega$-fancy and $d$-loose. Then there is a $c$-hassle in $G$.
\end{theorem}

  \begin{proof}
Let $\eta=\eta(c+1,t)$ be as in Theorem~\ref{thm:noshort}. Let $\zeta=\zeta(c,d,\eta,t)$ be as in Lemma~\ref{lem:bananinpoly}. We define
$$\Omega=\Omega(c,d,t)=\kappa(\max\{t^{\eta},\zeta\},t)+1,$$
where $\kappa(\cdot,\cdot)$ is as in Theorem~\ref{thm:tw7block}. Let $G$ be a $t$-clean graph and let $\mathcal{W}$ be a $2\Omega$-polypath in $G$ which is both $\Omega$-fancy and $d$-loose. Suppose for a contradiction that $G$ does not contain a $c$-hassle.

Let $H=G[V(\mathcal{W})]$. Then $H$ is a $t$-clean graph which has a $K_{\Omega,\Omega}$-minor. It follows that $H$ has treewidth at least $\Omega$, which along with Theorem~\ref{thm:tw7block} implies that there is a strong $\max\{t^{\eta},\zeta\}$-block in $G$. In particular, we may choose a $t^{\eta}$-subset $B$ of $V(H)$ such that for every $2$-subset $\{x,y\}$ of $B$, there exists a collection $\mathcal{P}_{\{x,y\}}$ of $\zeta$ pairwise internally disjoint paths in $H$ from $x$ to $y$. Since $|B|=t^{\eta}$ and since $G$ is $K_{t+1}$-free, it follows from Theorem~\ref{classicalramsey} that there exists a stable set $A\subseteq B$ in $H$ of cardinality $\eta$.

Now, fix a $2$-subset $\{x,y\}$ of $A$. Then $x,y$ are non-adjacent in $G$, and so by the choice of $\zeta$,  we can apply Lemma~\ref{lem:bananinpoly} to $\mathcal{W}$ and $\mathcal{P}_{\{x,y\}}$. Note that Lemma~\ref{lem:bananinpoly}\ref{lem:bananinpoly_a} violates the assumption that $G$ is $t$-clean, and Lemma~\ref{lem:bananinpoly}\ref{lem:bananinpoly_b} violates the assumption that $G$ contains no $c$-hassle. Thus, Lemma~\ref{lem:bananinpoly}\ref{lem:bananinpoly_c} holds, that is, there exists a collection $\mathcal{Q}_{\{x,y\}}$ of $\eta$ pairwise internally disjoint paths in $G$ from $x$ to $y$, each of length at most $c+1$. But now $A$ is a $(c+1)$-short (not necessarily strong) $\eta$-block in $G$. Combined with the choice of $\eta$ and Theorem~\ref{thm:noshort}, this violates the assumption that $G$ is $t$-clean, hence completing the proof of Theorem~\ref{thm:tasselinpoly}.
\end{proof}

Combining Theorem~\ref{thm:grid}, \ref{thm:tasselincluster} and \ref{thm:tasselinpoly}, we deduce Theorem~\ref{mainloaded}, restated below.

\setcounter{theorem}{0}
\begin{theorem}\label{mainloaded}
For all $c,t\in \poi$,  there is a constant $\Gamma=\Gamma(c,t)$ such that every $t$-clean graph $G$ of treewidth more than $\Gamma$ contains a $c$-hassle.
\end{theorem}
\begin{proof}
    Let  $l=2c^2t^t$ and let $s=2ct^t$. Let $\Omega=\Omega(c,l+t^ts^{t^t},t)$ be as in Theorem~\ref{thm:tasselinpoly}. Let
$$\Gamma=\Gamma(c,t)=\gamma(c,l,s,t,\Omega);$$
where $\gamma(\cdot,\cdot,\cdot,\cdot,\cdot)$ is as in Theorem~\ref{thm:grid}. Let $G$ be a $t$-clean graph of treewidth more than $\Gamma$. It follows from Theorem~\ref{thm:grid} that either there exists a $t^t$-meager $(s,l)$-cluster in $G$, or there is $2\Omega$-polypath in $G$ which is both $\Omega$-fancy and $(l+t^ts^{t^t})$-loose. In the former case, by Theorem~\ref{thm:tasselincluster}, $G$ contains a $c$-hassle. Also, in the latter case, the choice of $\Omega$ along with Theorem~\ref{thm:tasselinpoly} implies that $G$ contains a $c$-hassle. This completes the proof of Theorem~\ref{mainloaded}.
\end{proof}
 
\section{Part assembly}\label{sec:end}
Finally, let us bring everything together and prove Theorem~\ref{mainiftassel}. 
First, from Theorem~\ref{mainloaded}, we deduce that:
\begin{theorem}\label{mainif}
    Let $\mathcal{H}$ be a finite set of graphs which is hassled. Then the class of all $\mathcal{H}$-free graphs is clean.
\end{theorem}
\begin{proof}
    Let $\mathcal{H}$ be a finite set of graphs which is hassled. In order to show that the class of all $\mathcal{H}$-free graphs is clean, it is enough to prove, for every $t\in \poi$, that  there is a constant $\ell=\ell(t,\mathcal{H})\in \poi$ such that every $t$-clean graph of treewidth more that $\ell$ contains some graph $H\in \mathcal{H}$.

    We begin with setting the value of $\ell$. Since $\mathcal{H}$ is hassled, it follows that for every $t\in \poi$,  there is a constant $c=c(t,\mathcal{H})\in \poi$ such that every $(K_{t+1},K_{t,t})$-free $c$-hassle contains all components of some graph in $\mathcal{H}$. Let $\Gamma=\Gamma(c,t)$ be as in Theorem~\ref{mainloaded}, and let $\Delta=\Delta(||\mathcal{H}||, ||\mathcal{H}||,t)$ be as in Lemma~\ref{ramsey2}. Define 
$$\ell=\ell(t,\mathcal{H})=||\mathcal{H}||^2\Delta+\Gamma.$$
Let $G$ be a $t$-clean graph of treewidth more than $\ell$. Since $\ell\geq \Gamma$, by Theorem~\ref{mainloaded}, there is an induced subgraph $\Xi$ of $G$ which is a $c$-hassle. This, combined with the assumption that $\mathcal{H}$ is hassled and $G$ is $(K_{t+1},K_{t,t})$-free, implies that there exists $W\subseteq V(\Xi)\subseteq V(G)$ with $|W|\leq ||\mathcal{H}||$ such that $G[W]$ contains all components of some graph in  $\mathcal{H}$.

Let $m\in \poi$ be maximum such that there are $m$ pairwise disjoint subsets $W_1,\ldots, W_m$ of $V(G)$, each of cardinality at most $||\mathcal{H}||$, such that for every $i\in \poi_{m}$ there exists $H \in \mathcal{H}$ such that
 $W_i$  contains every component of $H$. 
We claim that:

\sta{\label{st:manycopies} $m\geq ||\mathcal{H}||\Delta$.}

Suppose not. Let $G'=G\setminus (W_1\cup \cdots\cup W_m)$. Then we have $\tw(G')\geq \tw(G)-m||\mathcal{H}||>\tw(G)-||\mathcal{H}||^2\Delta>\Gamma$. Thus, by Theorem~\ref{mainloaded}, there is an induced subgraph $\Xi'$ of $G'$ which is a $c$-hassle. Since $\mathcal{H}$ is hassled and $G$ is $(K_t,K_{t,t})$-free, it follows that there exists $W'\subseteq V(\Xi')\subseteq V(G')$ with $|W'|\leq ||\mathcal{H}||$  and $H \in \mathcal{H}$ such that $W'$ contains every
component of $H$. But this violates the maximality of $m$, and so proves \eqref{st:manycopies}.

\medskip

By \eqref{st:manycopies}, we can choose pairwise disjoint subsets $W_1,\ldots, W_{||\mathcal{H}||\Delta}$ of $V(G)$, each of cardinality at most $||\mathcal{H}||$, such that for every $i\in \poi_{||\mathcal{H}||\Delta}$, there exists
$H_i \in \mathcal{H}$ such that $W_i$ contains every component of $H_i$.
Since $|\mathcal{H}|\leq ||\mathcal{H}||$, it follows that there exists a graph $H\in \mathcal{H}$, as well as  $I\subseteq \poi_{||\mathcal{H}||\Delta}$ with $|I|=\Delta$, such that for every $i\in I$, we have $H_i=H$. Also, since $G$ is $(K_t,K_{t,t})$-free, from the choice of $\Delta$, it follows that there exist $i_1,\ldots, i_{||\mathcal{H}||}\in I$ such that $W_{i_1},\ldots, W_{i_{||\mathcal{H}||}}$ are pairwise anticomplete in $G$.

Let $K_1,\ldots, K_h$ be the components of $H$. Then $h\leq ||\mathcal{H}||$, and for every $j\in \poi_{h}$, there exists $X_j\subseteq W_{i_j}$ such that $G[X_j]$ is isomorphic to $K_j$. Now $G[X_1\cup \cdots \cup X_h]$ is isomorphic to $H$, as desired.
\end{proof}

Now Theorem~\ref{mainiftassel} is immediate from Theorems~\ref{thm:tassellediffhassled} and \ref{mainif}.

\section{Connection to unavoidable binary strings}

We conclude the paper with a brief discussion around explicit constructions of tasselled families. There is one particularly nice example of a tasselled family, described in the following corollary of Theorem~\ref{thm:onlyif} (we omit the proof as it is straightforward).

\begin{corollary}\label{cor:nicetasslled}
    Given $a, b\in \poi$, let $\mathcal{H}_{a, b}$ be the set of all
 strands $F$ such that the path of $F$ has vertices $v_1, \dots, v_{a+b+1}$ in this order, and the neck $x$ of $F$ is non-adjacent to $v_1, \dots, v_a$ and adjacent to $v_{a+1}$. (Thus, $\mathcal{H}_{a, b}$ contains $2^b$ distinct graphs, corresponding to all possible ways for $x$ to be adjacent to $v_{a+2}, \dots, v_{a+b+1}$.) Then $\mathcal{H}_{a, b}$ is tasselled.
\end{corollary}

While a full description of all tasselled families remains unknown, a promising approach is to attempt to model the ``tasselled'' property word-combinatorially, and then use results from the literature on formal languages in order to gain insight into the graph problem.

One way to view a strand $F$ is as a binary string denoting neighbors of the neck in the path. More explicitly, for a strand $F$ with neck $v$ and path $P$, if $p_1, \ldots, p_t$ are the vertices of $P$ in order, then we can define the binary string $S_{P,v}=b_1 \dots b_t$ where $b_i = 1$ if $p_i \in N(v)$, and $b_i = 0$ otherwise (note that the string $S_{P,v}$ is uniquely defined up to reversal). Moreover, every strand $F$ of a $c$-tassel with neck $v$ corresponds to a binary string (again, unique up to reversal) which starts and ends with at least $c$ zeroes, but which is not an all-zero string; let us say such a string is \emph{$c$-padded}. 

For a graph $K$, we say a vertex $v\in V(K)$ is \textit{a neck for $K$} if every component of $K\setminus \{v\}$ is a path, and we denote by $\mathcal{S}_{K,v}$ the set of all strings $S_{P,v}$ where $P$ is a component of $K\setminus \{v\}$. It is easy to see that: 
\begin{theorem}
    Let $\mathcal{H}$ be a set of graphs. Then $\mathcal{H}$ is tasselled if and only if there exists $c$ such that, for every $c$-padded string $S$, there is a graph $H\in \mathcal{H}$ with the following property. For every component $K$ of $H$, one may choose a neck $v$ of $K$ such that for every string $S'\in \mathcal{S}_{K,v}$, the string $S$ contains either $S'$ or its reverse as a consecutive substring. 
\end{theorem}

There is an interesting special case of this where every graph in $\mathcal{H}$ is a $c'$-tassel for some $c'\in \poi$. Note that in this case, every $H\in \mathcal{H}$ is connected and the set $\mathcal{S}_{H, v}$ has cardinality one for each neck $v$ of $H$. So the question of whether $\mathcal{H}$ is tasselled translates into:

\begin{question}
    For which sets $\mathcal{S}$ of strings is it true that there is a constant $c=c(\mathcal{S})\in \poi$ such that every $c$-padded string $S$ contains either a string in $\mathcal{S}$ or the reverse of a string in $\mathcal{S}$ as a consecutive substring? 
\end{question}

If the above holds for some fixed $c$, let us call such a set $\mathcal{S}$ a \emph{$c$-unavoidable language}. In view of our aim, it is of particular interest to identify the minimal $c$-unavoidable languages (with respect to the inclusion). In particular, the only sets of cardinality one which are $c$-unavoidable for some $c\in \poi$ are of the form $\mathcal{S} = \{0\cdots01\}$ (up to reversal) and $\mathcal{S} = \{0\cdots0\}$, and one may observe that this matches the outcome of Theorem~\ref{tw7}. However, in general, larger $c$-unavoidable languages seem harder to describe; for example, the following is $3$-unavoidable: 
$$\{00011, 0001000, 1010, 010010, 111, 110011, 11011, 110010011, 00010011001000\}.$$ On the other hand, for a finite set $\mathcal{S}$ of strings, it is finitely testable if $\mathcal{S}$ is $c$-unavoidable for some $c\in \poi$. To see this, note that if there is a $c$-padded string $S$ which does not contain any string in $\mathcal{S}$ nor the reversal of a string in $\mathcal{S}$, then, assuming $s$ to be the length of the longest string in $\mathcal{S}$, one may choose $S$ such that $c\leq s$ and $S$ has length at most $(s+2)2^s$ (which follows by noticing that a repetition of a consecutive substring of length $s$ in $S$ not containing the first or last $c$ bits can be used to shorten $S$). 

The following, closely related question, has been studied before: 

\begin{question}
 For which sets $\mathcal{S}$ of strings is it true that every sufficiently long string contains a string in $\mathcal{S}$ as a consecutive substring? 
\end{question}

These sets $\mathcal{S}$ are called \emph{unavoidable languages} (see, for example \cite{rosaz1995unavoidable, rosaz1998inventories}). In particular, in  \cite{bucher1985total, choffrut1984extendibility} it is shown that every minimal unavoidable language is finite. This is not the case for $c$-unavoidable languages; for example, $\mathcal{S} = \{01^k0 : k\in \poi\}$ is $1$-unavoidable. However, if we omit the string $01^k0$ for some $k\in \poi$, then $0\dots 01^k0 \dots 0$ avoids $\mathcal{S}$. It would be interesting to consider which results on unavoidable languages generalize to $c$-unavoidable languages.

\bibliographystyle{amsplain}


\begin{aicauthors}
\begin{authorinfo}[balec]
  Bogdan Alecu\\
  School of Computing, University of Leeds\\
  Leeds, UK
\end{authorinfo}
\begin{authorinfo}[mchud]
 Maria Chudnovsky\\
 Princeton University\\
 Princeton, New Jersey, USA
\end{authorinfo}
\begin{authorinfo}[laci]
Sepehr Hajebi\\
Department of Combinatorics and Optimization, University of Waterloo\\
Waterloo, Ontario, Canada
\end{authorinfo}
\begin{authorinfo}[andy]
 Sophie Spirkl\\
Department of Combinatorics and Optimization, University of Waterloo\\
Waterloo, Ontario, Canada
\end{authorinfo}
\end{aicauthors}

\end{document}